\documentclass[9pt,reqno]{amsart}

\usepackage[utf8]{inputenc}
\usepackage[english]{babel}
\usepackage{amsmath,amsfonts,amssymb,amsthm}
\usepackage[T1]{fontenc}
\usepackage{lmodern}
\usepackage{mathtools}

\usepackage[top=3.5cm,bottom=3.5cm,left=3.6cm,right=3.6cm]{geometry}

\usepackage[dvipsnames]{xcolor}
\usepackage[hyperindex=true,frenchlinks=true,colorlinks=true,
citecolor=Mahogany,linkcolor=DarkOrchid,urlcolor=Tan,linktocpage,
pagebackref=true]{hyperref}

\usepackage{tikz}
\usetikzlibrary{shapes}
\usetikzlibrary{fit}
\usetikzlibrary{decorations.pathmorphing}

\usepackage{cite}
\usepackage{enumitem}
\usepackage{multicol}
\usepackage{etex}

\numberwithin{equation}{section}
\renewcommand{\leq}{\leqslant}
\renewcommand{\geq}{\geqslant}

\tikzstyle{Vertex}=[circle,draw=LimeGreen!80,fill=LimeGreen!8,
inner sep=1pt,minimum size=2mm,line width=1pt,font=\scriptsize]
\tikzstyle{Node}=[Vertex,draw=RoyalBlue!80,fill=RoyalBlue!8,inner sep=1.5pt]
\tikzstyle{Leaf}=[rectangle,draw=Black!70,fill=Black!16,
inner sep=0pt,minimum size=1mm,line width=1.25pt]
\tikzstyle{Edge}=[Maroon!80,cap=round,line width=1pt]
\tikzstyle{Mark1}=[draw=BrickRed!80,fill=BrickRed!8]
\tikzstyle{Mark2}=[draw=BurntOrange!80,fill=BurntOrange!8]
\tikzstyle{EdgeRew}=[->,RedOrange!80,cap=round,thick]

\newcommand{\norm}[1]{\left\lVert #1 \right\rVert}

\newcommand \ens[1]{\left\{ #1\right\}}
\newcommand \R{\mathbb R}

\newcommand \N{\mathbb N}

\newcommand \Z{\mathbb Z}

\newcommand \abs[1]{\left|#1\right|}
\newcommand \eps{\varepsilon}

\newcommand{\F}{\mathcal F}

\newcommand{\prt}[1]{\left(#1\right)}
\newcommand{\pr}[1]{\left(#1\right)}
\newcommand{\E}[1]{\mathbb E\left[#1\right]}

\newcommand{\PP}{\mathbb P}
\newcommand{\gr}[1]{\mathbf{#1}}
\newcommand{\imd}{\preccurlyeq}
\newcommand{\smd}{\succcurlyeq}

\newcommand{\Hide}[1]{}

\title[Deviation inequalities martingales differences sequences and random fields]{Deviation 
inequalities for Banach space valued martingales differences sequences and random fields}
\keywords{Martingales, random fields, orthomartingales, deviation inequalities, 
complete convergence.}
\date{\today}
\author{Davide Giraudo}
\address{Ruhr-Universität Bochum
Fakultät für Mathematik
NA 3/32
Universitätsstraße 150
44780 Bochum‚.}
\email{Davide.Giraudo@ruhr-uni-bochum.de}

\renewcommand{\leq}{\leqslant}
\renewcommand{\geq}{\geqslant}

\newtheorem{Theorem}{Theorem}[section]
\newtheorem{Proposition}[Theorem]{Proposition}
\newtheorem{Lemma}[Theorem]{Lemma}
\newtheorem{Definition}[Theorem]{Definition}
\newtheorem{Corollary}[Theorem]{Corollary}

\theoremstyle{remark}
\newtheorem{Remark}[Theorem]{Remark}

\begin{document}

\begin{abstract}
We establish deviation inequalities for the maxima of partial sums of 
a martingale differences sequence, and of an orthomartingale differences random field. These 
inequalities can be used to give rates for linear regression and the law of large numbers. 
\end{abstract}

\maketitle 

\section{Introduction and main results}

Deviation inequalities play an important role in the study of properties 
of partial sums of random variables. A particular attention has been 
given to martingales. In Burkholder's paper \cite{MR0365692},
distribution function inequalities for maximum of martingales are established, 
and moment inequalities are derived from them. Sharp results has been obtained 
for martingales with bounded increments \cite{MR0144363,MR0221571}. When the 
increments of the considered martingale
are unbounded but square integrable, it is possible to control 
the tail function of the martingale by that of the increments 
and of the sum of conditional variances, like in 
\cite{MR0365692,MR736144,MR1681153,MR3005732,MR3311214}. 
When the tail of increments have a polynomial decay, it seems that 
Nagaev's inequality \cite{MR2021875} gives the most 
satisfactory results. It states the following: for any positive $q$, there exists a constant $C(q)$ 
such that if $\pr{S_n}_{n\geq 1}$ is a martingale defined on a probability space $\pr{\Omega,\mathcal F,\PP}$ and $X_i:=S_i-S_{i-1}$, then 
\begin{multline}\label{eq:Nagaev}
 \PP\ens{\abs{S_n}>x} \leq C(q)\int_0^1\PP\ens{\max_{1\leq i\leq n}\abs{X_i}>xu}u^{q-1}\mathrm du
 \\+C(q)\int_0^1\PP\ens{ \pr{\sum_{i=1}^n\E{X_i^2\mid \F_{i-1}}}>xu}u^{q-1}\mathrm du.
\end{multline}
The constant $C(q)$ is of order $e^{e^q }$. The result (without the absolute valued in the left hand side 
of \eqref{eq:Nagaev}) holds for supermartingales. There are three possibilities of improvement of the 
version of Nagaev's result for martingales:
\begin{itemize}
 \item the result can only be used for square integrable martingales.  One can wonder 
 whether a similar inequality as \eqref{eq:Nagaev} holds when $X_i\in\mathbb L^p$ where 
 $1<p<2$.
 \item In \cite{MR2021875}, the real valued case is considered, and the proof suggests that the 
 extension to the Banach valued case is challenging.
 \item Finally, the improvement of the constant $C(q)$ is also of interest.
\end{itemize}

Let us explain the idea of proof of an extension of \eqref{eq:Nagaev} (see 
Theorem~\ref{thm:deviation_inequality_non_stationary_martingale_diff}) in the real valued case, with square integrability. Define 
\begin{equation}
 f(x):=\PP\ens{\abs{S_n}>x}\mbox{ and }
\end{equation}
\begin{equation}
 g(x):=\PP\ens{\max_{1\leq i\leq n}\abs{X_i}>x} +\PP\ens{ \pr{\sum_{i=1}^n\E{X_i^2\mid 
 \F_{i-1}}}>x}, 
\end{equation}
We prove in Lemma~\ref{lem:functional_inequality} that for any positive $x$.
\begin{equation}\label{eq:etape_cle_idee_demo}
 f\pr{2x}\leq \delta^2\pr{1-\delta}^{-2}
f(x)+g\pr{\delta x}.
\end{equation}
This is done by using a martingale transform of the original martingale, the former having small conditional variances. 
Using monotonicity of the function $g$, \eqref{eq:etape_cle_idee_demo} can be converted 
into an integral inequality.

The paper is organized as follows: in Subsection~\ref{subsec:martingale_differences_sequences}, 
we state a deviation inequality for any Banach space valued martingale 
differences sequence, then for stochastically dominated or identically distributed sequences. In 
Subsection~\ref{subsec:orthomartingale_random_fields}, we
review orthomartingales, and state a deviation inequality 
for orthomartingale differences random fields. Section~\ref{sec:applications} 
is devoted to applications to linear regression
and Baum-Katz estimates martingale differences sequence 
and orthomartingale differences random fields. All these results are proven 
in Section~\ref{sec:proofs}.

  \subsection{Martingale differences sequences}
  \label{subsec:martingale_differences_sequences}

  \subsubsection{General case}
\begin{Definition}
 Let $\prt{\Omega,\mathcal F,\PP}$ be a probability space and 
let $\prt{B,\norm{\cdot}_{B}}$ be a separable Banach space. For any 
$p\geq 1$, we denote by $\mathbb L^p_B$ the space of $B$-valued 
random variables such that $\norm{X}_{\mathbb L^p_B}^p=\E{ \norm{X}^p  }$ 
is finite. Let $\prt{\F_i}_{i\geq 1}$ be an non-decreasing sequence of 
sub-$\sigma$-algebras of $\F$. We say that a sequence of $B$-valued 
random variables $\pr{X_i}_{i\geq 1}$ is a martingale differences 
sequence with respect to the filtration $\prt{\F_i}_{i\geq 1}$ if 
\begin{enumerate}
 \item for any $i\geq 1$, $X_i$ is $\F_i$-measurable and 
 belongs to $\mathbb L^1_B$;
 \item for any $i\geq 2$, $\E{X_i\mid \F_{i-1}}=X_{i-1}$ almost 
 surely.
\end{enumerate}
\end{Definition}

\begin{Definition}
 Following \cite{MR0394135}, we say that a Banach space 
 $\prt{B,\norm{\cdot}}$ is $r$-smooth ($1<r\leq 2$) 
 if there exists an equivalent norm $\norm{\cdot}'$ such that 
 \begin{equation*}
  \sup_{t>0}\frac 1{t^r}\sup\ens{ \norm{x+ty}'+\norm{x-ty}'-2: \norm{x}'=\norm{y}'=1  }<\infty.
 \end{equation*}
\end{Definition}

From \cite{MR0407963}, we know that if $B$ is $r$-smooth and separable, 
then there exists a constant $D$ such that for any sequence of 
$B$-valued martingale differences $\prt{X_i}_{i\geq 1}$,

\begin{equation}\label{eq:rD_smooth_martingale}
 \E{ \norm{\sum_{i=1}^nX_i}^r} \leq D  \sum_{i=1}^n\E{\norm{X_i}^r}.
\end{equation}
 
Since an $r$-smooth Banach space is also $r'$-smooth for any $1<r'\leq r$, there exists a 
constant $C_{r',B}$ such that 
such that for any sequence of 
$B$-valued martingale differences $\prt{X_i}_{i\geq 1}$, and any integer $n$, 

\begin{equation}\label{eq:r'_smooth_martingale}
 \E{ \norm{\sum_{i=1}^nX_i}^{r'}} \leq C_{r',B}  \sum_{i=1}^n\E{\norm{X_i}^{r'}}.
\end{equation}

Our first main result is an inequality in the spirit of  
Theorem~1 in \cite{MR2021875}.

  \begin{Theorem}\label{thm:deviation_inequality_non_stationary_martingale_diff}
   Let $\pr{B,\norm{\cdot}}$ be a separable $r$-smooth Banach space where 
   $1<r\leqslant 2$. 
   For each $1<r'\leqslant r$, $q>0$ and for any $B$-valued martingale 
   differences sequence $\pr{X_i,\F_i}_{i\geq 1}$, the following inequality 
   holds for each $n\geq 1$ and $x>0$:
   \begin{multline}\label{eq:deviation_inequality_non_stationary_martingale_diff}
    \PP\ens{\max_{1\leqslant i\leqslant n}\norm{S_i}>x }
    \leq \frac{2^q}{2^q-1}q2^{-r'}\int_0^1\PP\ens{\max_{1\leqslant i\leqslant n}
    \norm{X_i}>2^{-1-q/r'}C_{r',B}^{-1/r'}xu}u^{q-1}\mathrm du\\+\frac{2^q}{2^q-1}q2^{-r'}\int_0^1\PP\ens{
    \pr{\sum_{i=1}^n
    \E{\norm{X_i}^{r'}\mid \F_{i-1}}}^{1/r'}>2^{-1-q/r'}C_{r',B}^{-1/r'}xu}u^{q-1}\mathrm du,
   \end{multline}
   where $S_i=\sum_{i=1}^nX_i$ and $C_{r',B}$ is a constant satisfying 
   \eqref{eq:r'_smooth_martingale} for any $n$ and any martingale differences sequence.
  \end{Theorem}
  
  \begin{Remark}
  On one hand, Nagaev's result \cite{MR2021875} applies to real valued supermartingales, while 
  our result is restricted to martingales. On the other hand, when applied to the latter 
  class of random variable, our result gives a generalization in two directions. First, 
  we consider Banach space valued random variables. Second, even when restricted to 
  real-valued random variables, our result can be used to treat martingales whose 
  increments do not necessarily have a finite moment of order $2$. 
  \end{Remark}
  
  In the independent setting, the terms $\E{\norm{X_i}^{r'}\mid \F_{i-1}}$ are constant hence 
  we can state the following Corollary of 
  Theorem~\ref{thm:deviation_inequality_non_stationary_martingale_diff}.
  
  \begin{Corollary}
   Let $\pr{B,\norm{\cdot}}$ be a separable $r$-smooth Banach space where 
   $1<r\leqslant 2$. 
   For each $1<r'\leqslant r$, $q>0$ and for any independent centered
    sequence $\pr{X_i}_{i\geq 1}$, the following inequality 
   holds for each $n\geq 1$ and $x>0$:
   \begin{multline}\label{eq:deviation_inequality_independent}
    \PP\ens{\max_{1\leqslant i\leqslant n}\norm{S_i}>x }
    \leq \frac{2^q}{2^q-1}q2^{-r'}\int_0^1\PP\ens{\max_{1\leqslant i\leqslant n}
    \norm{X_i}>2^{-1-q/r'}C_{r',B}^{-1/r'}xu}u^{q-1}\mathrm du\\+C_{r',B}^{(q-1)/r'}\frac{2^q}{2^q-1}2^{-r'}
    2^{q+q^2/r'}x^{-q}\pr{\sum_{i=1}^n
    \E{\norm{X_i}^{r'}}}^{q/r'},
   \end{multline}
   where $S_i=\sum_{i=1}^nX_i$ and $C_{r',B}$ is a constant satisfying 
   \eqref{eq:r'_smooth_martingale} for any $n$ and any martingale differences sequence.
  \end{Corollary} 
  
   \subsubsection{Stochastically dominated sequences}
  
  For a random variable $Y$ with values in the Banach space $\pr{B,\norm{\cdot}}$, 
  we denote by $Q_Y$ the generalized inverse of the function $t\mapsto 
  \PP\ens{\norm{Y}>t}$, that is, 
  \begin{equation}
   Q_Y\pr{u}:=\inf\ens{t>0 \mid \PP\ens{\norm{Y}>t} \leq u}, \quad u \in [0,1].
  \end{equation}

  \begin{Definition}
  Let $\pr{X_i}_{i\geq 1}$ be a sequence of random variables with values in a Banach 
  space $\pr{B,\norm{\cdot}}$ and let $X\colon \Omega\to \R$ be a real valued 
  random variable. We say that $\pr{X_i}_{i\geq 1}\prec X$ if for all $u\in [0,1]$, 
  and any $i\geq 1$, $Q_{X_i}\pr{u}\leq Q_X\pr{u}$.
  \end{Definition}

  In the case where the random variables $X_i$, $1\leq i\leq n$ are stochastically 
  dominated and $\E{\norm{X_i}^{r'}\mid \F_{i-1}}$ bounded by identically distributed random 
  variables, the result of 
  Theorem~\ref{thm:deviation_inequality_non_stationary_martingale_diff} 
  admits the following simplification.
  
  \begin{Theorem}\label{thm:deviation_inequality_stationary_martingale_diff}
   Let $\pr{B,\norm{\cdot}}$ be an $r$-smooth separable 
   Banach space. For each $1<r'\leqslant r$ and each $q>r'$, for any 
   martingale differences sequence $\prt{X_i, \F_i}_{i\geq 0}$
   with values in $B$ such that there exists real valued random variables 
   $X$ and $V_i$, $i\geq 1$ for which
   $\pr{X_i}_{i\geq 1}\prec X$, 
   $\E{\norm{X_i}^{r'}\mid\F_{i-1}}\leq V_i$ a.s. and $\pr{V_i}_{i\geq 1}$ 
   is identically distributed, then 
    the following inequality hold for any $n\geq 1$ and $x>0$:
   \begin{multline}\label{eq:deviation_inequality_stationary_martingale_diff_cond_var}
    \PP\ens{\max_{1\leqslant i\leqslant n}\norm{S_i}>xn^{1/r'} }
    \leqslant \frac{2^q}{2^q-1}q2^{-r'}n\int_0^1\PP\ens{
    X>2^{-1-q/r'}C_{r',B}^{-1/r'}xun^{1/r'}}u^{q-1}\mathrm du\\+
     \frac{2^q}{2^q-1}\frac{q2^{-r'}}{q-r'}\int_0^{+\infty}
    \PP\ens{V_1>2^{-2-q/r'}C_{r',B}^{-1/r'}x^{r'}w}
    \min\ens{w^{\frac{q-r'}{r'}},1}\mathrm dw.
   \end{multline}
   If the sequence $\pr{\norm{X_i}}_{i\geq 1}$ is identically distributed, then 
    for any $n\geq 1$ and $x>0$:
   \begin{equation}\label{eq:deviation_inequality_stationary_martingale_diff}
    \PP\ens{\max_{1\leqslant i\leqslant n}\norm{S_i}>xn^{1/r'} }
    \leqslant  \frac{2^{q+1}}{2^q-1}\frac{q2^{-r'}}{q-r'}\int_0^{+\infty}
    \PP\ens{\norm{X_1}>xu}\min\ens{u^{q-1},u^{r'-1}}\mathrm du,
   \end{equation}
   where $S_i=\sum_{j=1}^iX_j$.
  \end{Theorem}

  
  \subsection{Orthomartingale differences random fields}
  \label{subsec:orthomartingale_random_fields}

  The results of the previous section can be extended in some 
  sense to random fields, that is, processes indexed by $\N^d$ or 
  $\Z^d$ where $d\geq 1$ be an integer. In order to state them, we have to give a precise definition of 
  martingales in this setting.
  
  We use the following notations:
  \begin{enumerate}
  \item for $\gr{i}=\pr{i_q}_{q=1}^d$ and $\gr{j}=\pr{j_q}_{q=1}^d$ we write $\gr{i}\imd \gr{j}$ 
  if and only if $i_q\leq j_q$ for all $q\in \ens{1,\dots,d}$;
   \item if $\gr k$ and $\gr l\in \Z^d$ the coordinatewise minimum is defined by 
  $\min\ens{\gr{k},\gr{l}}=\prt{ \min\ens{k_i,l_i}}_{i=1}^d$.
  \item The addition is defined coordinatewise.
  \item If $\gr{n}=\pr{n_q}_{q=1}^d$ is an element of $\N^d$, then 
  $\abs{\gr{n}}$ denotes $\prod_{q=1}^dn_q$.
  \item For $j\in\ens{1,\dots,d}$, $\gr{e_j}$ denotes the element of $\Z^d$ 
  whose $j$-th coordinate is $1$ and all the others are zero. Moreover, $\gr{1}$ is the 
  element of $\Z^d$ whose all coordinates are $1$.
  \end{enumerate}

  \begin{Definition}
   The family $\pr{\F_{\gr{i}}}_{\gr{i}\in\Z^d}$ of sub-$\sigma$-algebras of $\F$ 
   is a filtration if $\F_{\gr{i}} \subset \F_{\gr{j}}$ whenever $\gr{i}\imd 
   \gr{j}$.
  \end{Definition}

  \begin{Definition}
  Let $\pr{\F_{\gr{i}}}_{\gr{i}\in\Z^d}$ be a filtration.
   If for each $\gr{i},\gr{j}\in \Z^d$ and each integrable random 
variable $Y$, 
 \begin{equation}
  \E{\E{Y \mid  \F_{\gr{i}}} \mid \F_{\gr{j}}}
  =\E{\E{Y \mid  \F_{\gr{j}}} \mid \F_{\gr{i}}}
  =\E{Y\mid \F_{\min\ens{\gr{i},\gr{j}}}}
  \mbox{ almost surely},
 \end{equation}
the filtration $\pr{\F_{\gr{i}}}_{\gr{i}\in\Z^d}$ is said to be commuting.
  \end{Definition}

 \begin{Definition}
  The collection of random variables $\ens{M_{\mathbf n},\mathbf n\in \Z^d}$ is 
said to be an orthomartingale random field with respect to the commuting
 filtration 
$\left( \F_{\gr{i}}\right)_{\mathbf i\in \Z^d}$ if 
for each $\mathbf n\in \N^d$, $M_{\mathbf n}$ is 
$\F_{\mathbf n}$-measurable, integrable 
and for each  $\mathbf i,\mathbf j\in \Z^d$ such that  $\mathbf i\preccurlyeq \mathbf j$, 
  \begin{equation}
   \mathbb E\left[M_{\mathbf j}\mid\F_{\mathbf i}\right]=M_{\mathbf i}.
  \end{equation}
 \end{Definition}
 
 \begin{Definition}
   The collection of random variables $\pr{X_{\gr{i}}}_{\gr{i}\in \Z^d}$ is 
said to be an orthomartingale diffrences random field with respect to the commuting
 filtration 
$\left( \F_{\gr{i}}\right)_{\mathbf i\in \Z^d}$ if the random field $\pr{S_{\gr{n}}}_{\gr{n}\in\Z^d}$ 
defined by 
\begin{equation}\label{eq:definitio_of_partial_sums}
 S_{\gr{n}}:=\begin{cases}
              \sum_{\gr{1}\imd\gr{i}\imd\gr{n}}X_{\gr{i}}&\mbox{if }\gr{n}\smd \gr{1},\\
              0&\mbox{ otherwise,}
             \end{cases}
\end{equation}
is an orthomartingale random field with respect to the filtration 
$\left( \F_{\gr{i}}\right)_{\mathbf i\in \Z^d}$.
 \end{Definition}

In all this subsection, we shall make the following assumption on the random field 
$\pr{X_{\gr{i}}}_{\gr{i}\in \Z^d}$, namely:
\begin{equation}\label{eq:stationarity_of_sums}
\mbox{for all }\gr{n}\smd \gr{1}, \gr{l}\in\Z^d, 
 \norm{\sum_{\gr{1}\imd\gr{i}\imd\gr{n}}X_{\gr{i}}} \mbox{ and }
  \norm{\sum_{\gr{1}\imd\gr{i}\imd\gr{n}}X_{\gr{i}+\gr{l}}} \mbox{ have the same distribution.}
\end{equation}

  Orthomartingale random fields have good properties with respect to 
  marginal filtrations $\mathcal F_q^{(d)}:=\sigma\left(\mathcal F_{\mathbf k},
  k_q\leqslant j,\mathbf k\in \Z^d\right)$, 
  $q\in \ens{1,\dots,d}$. Furthermore, when a coordinate is 
  fixed, we still have an orthomartingale random field with respect to a 
  commuting filtration (see \cite{MR1914748}, p.37, Theorem~3.5.1).

  \begin{Lemma}\label{lem:properties_orthomartingales}
  Let $\left(X_{\mathbf i}\right)_{\mathbf i\in \Z^d}$ be an
orthomartingale difference random field with respect to the commuting filtration 
$\left( \F_{\gr{i}}\right)_{\mathbf i\in \Z^d}$, with values in a separable Banach space 
$\pr{X,\norm{\cdot}}$. Then the following properties hold.
\begin{enumerate}[label=(P.\arabic*)]
 \item For any $\mathbf n=(n_1,\dots,n_{d-1})\in \N^{d-1}$, the sequence 
 $\left(S_{(\mathbf n,j)} \right)_{j\geqslant 0}$ is a martingale with respect to the 
 filtration $\left(\mathcal F_j^{(d)}\right)_{j\geqslant 0}$.
 
 \item\label{submartingale} For any $\mathbf n=(n_1,\dots,n_{d-1})\in \N^{d-1}$, the sequence 
 $\left(\max_{\mathbf i\in [\mathbf 1,\mathbf n]}
 \norm{S_{(\mathbf i,j)} }\right)_{j\geqslant 0}$ is a non-negative submartingale
  with respect to the 
 filtration $\left(\mathcal F_j^{(d)}\right)_{j\geqslant 0}$. 
 
 \item\label{fixed_coordinate} For any $j\in \mathbb N$, the random field 
 $\left(S_{(\mathbf n,j)} \right)_{\mathbf n\succcurlyeq 
 \mathbf 1}$ is an orthomartingale with respect to the commuting filtration 
 $\left( \F_{\gr{i}}'\right)_{\mathbf i\in \Z^{d-1}}$, where 
 $\F_{\gr{0}}'$ is the 
 $\sigma$-algebra generated by $\bigcup_{l\in \Z}\F_{l\gr{e_{d}}}$.
\end{enumerate}
 \end{Lemma}
 
  We now state the analogue of Theorem~\ref{thm:deviation_inequality_stationary_martingale_diff} 
  for strictly stationary orthomartingale differences random fields.

 \begin{Theorem}\label{thm:deviation_inequality_stationary_orthomartingale_diff}
   Let $\pr{B,\norm{\cdot}}$ be an $r$-smooth separable 
   Banach space. For each $1<r'\leqslant r$, $q>r'$ and positive integer $d$, there 
   exist a constant 
   $C$ depending only on $r'$, $q$, $d$ and $B$ such that for any strictly stationary 
   orthomartingale differences random field $\prt{X_{\gr{i}}, 
    \mathcal F_{\gr{i}}}_{i\in\Z^d}$
   with values in $B$, the following inequality holds for any $\gr{n}\in \N^d$ and $x>0$:

   \begin{equation}\label{eq:deviation_inequality_stationary_orthomartingale_diff}
    \PP\ens{\max_{\gr{1}\preccurlyeq  \gr{i}\preccurlyeq \gr{n}}
    \norm{S_{\gr{i}}}>x\abs{\gr{n}}^{1/r'} }\\
    \leq  C\int_0^{+\infty}
    \PP\ens{\norm{X_1}>xu}\min\ens{u^{q-1},u^{r'-1}}\prt{1+\abs{\log u}}^{d-1}
    \mathrm du.
   \end{equation}
  \end{Theorem}

  \begin{Remark}
   One can integrate the previously obtained inequalities to get moment inequalities. 
   For example, it is possible to recover a multidimensional Burkholder-like inequality 
   in the stationary case, like in \cite{MR2264866}. Like in the one dimensional 
   case, it is also possible to establish inequalities in weak $\mathbb L^p$ spaces 
   like in \cite{MR995572}, Remark~6.
  \end{Remark}

\section{Applications}
\label{sec:applications}

\subsection{Linear regression}

 We consider the stochastic linear regression model given by 
 \begin{equation}
  X_k=\theta \phi_k+\eps_k,\quad 1\leq k\leq n,
 \end{equation}
where 
\begin{itemize}
 \item $\pr{X_k}_{1\leq k\leq n}$ are the observations,
 \item $\pr{\phi_k}_{1\leq k\leq n}$ are the regression variables and 
 \item $\pr{\eps_k}_{1\leq k\leq n}$ the driven noises.
\end{itemize}

We shall make the following assumptions:
\begin{enumerate}[label=(A.\arabic*)]
 \item\label{asum:hyp_1} the sequence $\pr{\phi_k}_{1\leq k\leq n}$ is independent;
 \item\label{asum:hyp_2} the $\sigma$-algebra generated by $\phi_k$, $1\leq k\leq n$ 
 is independent of the  $\sigma$-algebra generated by $\eps_k$, $1\leq k\leq n$;
 \item\label{asum:hyp_3} For each $k\in\ens{2,\dots,n}$, $\E{\eps_k\mid \sigma\pr{\eps_i, 
 1\leq i\leq k-1}}=0$ and $\E{\eps_1}=0$.
\end{enumerate}
Let $\theta_n$ be the least square estimator defined by 
\begin{equation}
 \theta_n:=\frac{\sum_{k=1}^n\phi_kX_k}{\sum_{i=1}^n\phi_i^2}.
\end{equation}
\begin{Theorem}\label{thm:linear_regression}
 Suppose that the assumptions \ref{asum:hyp_1}, \ref{asum:hyp_2} and 
 \ref{asum:hyp_3} hold. Suppose that there exists constant $C_1$ and $C_2$ such that 
 for any $i\in\ens{1,\dots,n}$, 
 \begin{equation}
 \E{\abs{\eps_i}^p}\leq C_1 \mbox{ and }
 \E{\eps_i^2\mid\sigma\pr{\eps_j, 1\leq j\leq i-1}}\leq C_2\mbox{ a.s}.
 \end{equation}
Then for any $p>2$, $q>p$ and any $x>0$, 
\begin{equation}
 \PP\ens{\abs{\theta_n-\theta}\sqrt{\sum_{i=1}^n\phi_i^2}>x}\leq 
  C_1\frac{2^{q-2}}{2^q-1}\frac{q}{q-p}  2^{p+pq/2} x^{-p}+
  \frac{2^{q-2}}{2^q-1}q 2^{q+q^2/2}x^{-q}C_2^{q/2}.
\end{equation}

\end{Theorem}

Let us compare Theorem~\ref{thm:linear_regression} with the results in \cite{MR3579898}. 
When $x$ is large, Theorem~\ref{thm:linear_regression} and Theorems~3.3 and 3.4 in 
 \cite{MR3579898} give an upper bound of order $x^{-p}$.
\begin{enumerate}
 \item In Theorem~3.3 of \cite{MR3579898}, it is assumed that 
  $\sup_i\norm{\E{\abs{\eps_i}^p\mid \sigma\ens{\eps_k,1\leq k\leq i-1}}}_{\infty}<+\infty$,
   which is more restrictive than the assumption in Theorem~\ref{thm:linear_regression}.
 \item  In Theorem~3.4 of \cite{MR3579898}, $\sup_i\norm{\E{\abs{\eps_i}^2\mid \sigma
 \ens{\eps_k,1\leq k\leq i-1}}}_{\infty}<+\infty$ and that there exists a positive $\delta$ and $C_1$ 
 such that for all $i\geq 1$, $ \E{\abs{\eps_i}^{p+\delta}}\leq C_1$, which is more restrictive than 
 our result, since only boundedness of the sequence of moments of order $p$ is required.
\end{enumerate}

\subsection{Baum-Katz estimates for martingale differences sequences and
orthomartingale differences 
random fields}

\subsubsection{Martingale differences sequences}
\label{subsubsec:large_deviations_martingales}

For $p>1$, we denote by $\mathbb L^{p,\infty}$ (respectively 
$\mathbb L^{p,\infty}_0$) the set of random variables 
$X$ such that $\sup_{t>0}t^p\PP\ens{ \norm{X}>t}<+\infty$
(respectively $\lim_{t\to +\infty}t^p\PP\ens{ \norm{X}>t}=0$). 
We also write $\mathbb L^p\log^q\mathbb L$ (with $q\geq 0$) the set of random variables 
$X$ such that $\E{\norm{X}^p\pr{\log_+\norm{X} }^q     }$ is finite, where 
$\log_+\pr{x}:=\max\ens{0,\log\norm{x} }$.

\begin{Theorem}\label{thm:large_deviation_martingale_stationary_p_leq_2}
Let $B$ be an $r$-smooth Banach space for $1<r\leq 2$. Let $\pr{X_i}_{i\geq 1}$ be a martingale 
differences sequence with values in $B$. Assume that one of the following conditions is satisfied:
\begin{enumerate}[label=(C.\arabic*)]
\item\label{itm:condition_Baum_Katz_r_lisse_dom} there exists a real valued random variable $X$ in $\mathbb 
L^r$ such that $\pr{X_i}_{i\geq 1}
\prec X$ and there exists an identically distributed sequence $\pr{V_i}_{i\geq 1}$ such that for all $i$, 
$\E{\norm{X_i}^r\mid \F_{i-1}  } \leq V_i$ a.s. and $V_i\in \mathbb L\log\mathbb L$.
\item\label{itm:condition_Baum_Katz_r_lisse_id_distri} The sequence $\pr{\norm{X_i}}_{i\geq 1}$ is identically 
distributed and $X_1\in \mathbb L^r\log \mathbb L$.
\end{enumerate} 
Then for each $\alpha \in (1/r, 1]$ and each positive $x$, the series
$\sum_{n=1}^{+\infty}n^{r\alpha -2}\PP\ens{\max_{1\leq i\leq n}\norm{S_i}>n^{\alpha}x}$ 
 converges.
\end{Theorem}

Let us compare this result with a previous one.
In \cite{MR2743029}, convergence of the series 
\begin{equation*}
\sum_{n=1}^{+\infty}n^{p\alpha -2}\PP\ens{\max_{1\leq i\leq n}\norm{S_i}>n^{\alpha}x}
\end{equation*}

 have been established for sequences satisfying $\pr{X_i}_{i\geq 1}
\prec X$ and $X\in\mathbb L^p$ for $1<p<r$ and $1\leq\alpha\leq p$.
Our result deal with a more restrictive class of martingale differences but covers the case $p=r$.

When a moment of order greater than two is finite, we can formulate precise 
results in terms of integrability of the increments and of the conditional 
variance term.

\begin{Theorem}\label{thm:large_deviation_martingale_stationary_p_geq_2}
Let $p>2$, $1/2<\alpha\leq 1$ and let $B$ be a
separable $2$-smooth Banach space. There exists 
 a constant $C\prt{p,B}$ such that the following holds:
for each $B$-valued martingale 
differences sequence
$\prt{X_i, \F_i}$ such that $\pr{X_i}\prec X$ and such that there exists an identically distributed sequence 
$\pr{V_i}_{i\geq 1}$ for which $\E{X_i^2\mid \F_{i-1}}\leq V_i$,
\begin{enumerate}
 \item \label{itm:weak_moment}
 if $X\in \mathbb L^{p/2+1,\infty}$ and 
 $V_i\in \mathbb L^{p/2,\infty}$, then for each $x>0$, 
 \begin{multline}
 \sup_{n\geq 1} 
 n^{p\pr{\alpha -1/2}}\PP\ens{\max_{1\leq i\leq n}\norm{S_i}>n^{\alpha}x}
\\  \leq C\prt{p,B}\prt{\sup_{t>0}t^{p/2+1}\PP\ens{ \norm{X_1}>t}x^{-p/2-1}+ 
  \sup_{t>0}t^{p/2}\PP\ens{V_1>t}x^{-p}
  };
 \end{multline}
 \item\label{itm:weak_moment_bis} if $X_1 \in \mathbb L_0^{p/2+1,\infty}$ and 
 $V_1\in \mathbb L^{p/2,\infty}$, then for each 
 $x>0$,
 \begin{equation}
  \lim_{n \to +\infty}n^{p\pr{\alpha -1/2}}\PP\ens{\max_{1\leq i\leq n}\norm{S_i}>n^{\alpha}x}=0;
 \end{equation}
 \item\label{itm:strong_moment} if $X_1\in\mathbb L^{p/2+1}$ and $V_1
 \in \mathbb L^{p/2}$, then 
 \begin{equation}
  \sum_{n=1}^{+\infty}n^{p\pr{\alpha -1/2}-1}\PP\ens{\max_{1\leq i\leq n}\norm{S_i}>n^{\alpha}x}
  \leq C\prt{p,B} x^{-p}
  \prt{\norm{X_1}_p^p+\norm{V_1}_{p/2}^{p/2}} .
 \end{equation}
\end{enumerate}
\end{Theorem}

We can formulate an analogous result for "norm-identically" distributed sequences.

\begin{Theorem}\label{thm:large_deviation_id_distrib_p_geq_2}
Let $p>2$, $1/2<\alpha\leq 1$ and let $B$ be a
separable $2$-smooth Banach space. There exists 
 a constant $C\prt{p,B}$ such that the following holds:
for each $B$-valued martingale 
differences sequence
$\prt{X_i, \F_i}$ such that $\pr{\norm{X_i}}_{i\geq 1}$ is identically distributed,
\begin{enumerate}
 \item \label{itm:weak_moment_id_distrib}
 if $X_1\in \mathbb L^{p,\infty}$ then for each $x>0$, 
 \begin{equation}
 \sup_{n\geq 1} 
 n^{p\pr{\alpha -1/2}}\PP\ens{\max_{1\leq i\leq n}\norm{S_i}>n^{\alpha}x}
  \leq C\prt{p,B} \sup_{t>0}t^{p}\PP\ens{ \norm{X_1}>t}x^{-p/2-1} ;
 \end{equation}
 \item\label{itm:weak_moment_bis_id_distrib} if $X_1 \in \mathbb L_0^{p,\infty}$ then for each 
 $x>0$,
 \begin{equation}
  \lim_{n \to +\infty}n^{p\pr{\alpha -1/2}}\PP\ens{\max_{1\leq i\leq n}\norm{S_i}>n^{\alpha}x}=0;
 \end{equation}
 \item\label{itm:strong_moment_id_distrib} if $X_1\in\mathbb L^{p}$ then 
 \begin{equation}
  \sum_{n=1}^{+\infty}n^{p\pr{\alpha -1/2}-1}\PP\ens{\max_{1\leq i\leq n}\norm{S_i}>n^{\alpha}x}
  \leq C\prt{p,B} x^{-p}\norm{X_1}_p^p.
 \end{equation}
\end{enumerate}
\end{Theorem}

For condition~\ref{itm:condition_Baum_Katz_r_lisse_id_distri} to be satisfied, we require 
  $X_1\in\mathbb L^2 \log \mathbb L$ rather than 
in $\mathbb L^2$. One may wonder whether that stronger condition is really needed. 
Its "necessity" for stationary martingale differences sequences is proved below. It also show 
that the result of Theorems~\ref{thm:large_deviation_martingale_stationary_p_geq_2} and 
\ref{thm:large_deviation_id_distrib_p_geq_2} do not hold for $p=2$.

Recall that $\pr{X,\Sigma,\PP,\theta}$ is a dynamical system if $\pr{X,\Sigma,\PP}$ is a probability 
space and $\theta\colon X\to X$ is a measurable map such that $\PP\pr{\theta^{-1}A}=
\PP\pr{A}$ for all $A\in\Sigma$. If $f\colon\Omega\to \R$ is measurable, then the 
sequence $\pr{f\circ \theta^i}_{i\geq 1}$ is strictly stationary.

We start with the following lemma. 

\begin{Lemma}\label{lem:lemme-tech}
Let $\gamma>1$. There exist a dynamical system $\pr{X,\Sigma,\PP,\theta}$ and a non-negative measurable function 
$f\geq 0$ on $X$ such that, for  every $0<\varepsilon\leq 1$, $\int_Xf \pr{\log^+\pr{f}}^{1-\varepsilon}d\PP<\infty$ and 
\begin{equation}\label{eq:test}
\sum_{n\geq 0} 2^{n\pr{\gamma -1}} \PP\ens{  \sum_{i=0}^{2^n-1} f\circ\theta^{i}> 
2^{n\gamma }}=+\infty.
 \end{equation}
\end{Lemma}

\begin{Proposition}\label{prop:optimalite_cas_reel}
Let $\alpha>1/2$. There exists a stationary (and ergodic) sequence 
of martingale differences $\pr{X_i}_{i\geq 1}$ such that for every $0<\varepsilon\leq1$,
 $\E{X_1^2\pr{\log^+\abs{X_1}}^{1-\varepsilon}}<\infty$ and the series
$\sum_{n\geq 1} 2^{n\pr{2\alpha-1}}
\PP\ens{\abs{\sum_{i=1}^{2^n}X_i}> 2^{n\alpha}}$ diverges.
\end{Proposition}

In \cite{MR3194512}, Baum-Katz type estimates have been formulated for martingales differences 
arrays, extending the results in \cite{MR1084974}. It has been extended to the Banach space valued 
setting in \cite{MR3132544}. However, it seems that our results cannot be compared with those of 
\cite{MR3194512} because these ones require a control in of the $\mathbb{L}^p$-norm of $n^{-1}
\sum_{i=1}^n\E{\abs{X_i}^\gamma\mid \F_{i-1}}$.

\begin{Remark}
A similar statement holds when $n^{p\pr{\alpha -1/2}-1}\PP\ens{\max_{1\leq i\leq n}\norm{S_i}
>n^{\alpha}x}$ is replaced  by $2^{n\pr{p\pr{\alpha -1/2}}}\PP\ens{\max_{1\leq i\leq 2^n}
\norm{S_i}>2^{n\alpha}x}$ in item (3) of Theorems~~
\ref{thm:large_deviation_martingale_stationary_p_leq_2} and 
\ref{thm:large_deviation_id_distrib_p_geq_2}
 In view of Theorem~3.7 in \cite{MR1856684} when $\alpha=1$, the weight 
 $2^{np/2}$ is optimal for 
 stationary ergodic martingale differences sequences in the following sense: if 
 $\prt{R_n}_{n\geq 1}$ is a sequence of real numbers with goes to infinity and $p>2$, then 
 there exists a stationary martingale differences sequence $\prt{X_i}_{i\geq 0}$ 
 such that $X_1$ belongs to $\mathbb L^p$ but the sequence 
 $\prt{2^{np/2}R_n\PP\ens{\max_{1\leq i\leq 2^n}\abs{S_i}>2^n}}_{n\geq 1}$ 
 does not converge to $0$.
 \end{Remark}

\subsubsection{Orthomartingale differences 
random fields}
\label{subsubsec:OMDRF}

Let $\pr{X_{\gr{i}}}_{\gr{i}\in\Z^d}$ be an i.i.d. real-valued random field. Theorem~4.1 in 
\cite{MR0494431} gives the equivalence between the following two assertions for 
$\alpha>1/2$ and $p\geq \max\ens{1/\alpha,1}$:
\begin{enumerate}
\item $X_{\gr{1}}$ belongs to $\mathbb L^p\log^{d-1}\mathbb L$;
\item for each positive $\varepsilon$, 
\begin{equation} 
\sum_{\gr{n}\in\N^d}\abs{\gr{n}}^{p\alpha-2}\PP\ens{\max_{\gr{1}\imd \gr{i}\imd\gr{n}}
\norm{S_{\gr{i}}}>\varepsilon \abs{\gr{n}}^{\alpha}}<+\infty.
\end{equation}
\end{enumerate}

Deviation inequalities has been used in \cite{MR2794415,MR3451971} for 
the question of complete convergence of orthomartingale differences random fields.

Similar results as in Subsubsection~\ref{subsubsec:large_deviations_martingales} 
can be proved for some orthomartingale differences random fields.

\begin{Theorem}\label{thm:large_deviation_orthomartingale_stationary_p_leq_2}
 Let $B$ be a separable $r$-smooth Banach space. For each $B$-valued orthomartingale differences 
 random field $\pr{X_{\gr{i}}}_{\gr{i}\in \Z^d}$ satisfying \eqref{eq:stationarity_of_sums} and 
 such that $X_{\gr{1}}\in \mathbb L^r\log^d\mathbb L$, for each positive $\varepsilon$ and each 
 $\alpha\in \left(1/r,1\right]$, 
\begin{equation}\label{eq:Baum_Katz_champs}
\sum_{\gr{n}\in\N^d}\abs{\gr{n}}^{r\alpha-2}\PP\ens{\max_{\gr{1}\imd \gr{i}\imd\gr{n}}
\norm{S_{\gr{i}}}>\varepsilon \abs{\gr{n}}^{\alpha}}<+\infty.
\end{equation}
\end{Theorem}
\begin{Remark}
One could also formulate the corresponding result where $r$ is replaced in \eqref{eq:Baum_Katz_champs} by $1<p<r$. But this could be established in a more general context 
than ours, namely, that of stochastically dominated orthomartingale differences random fields, 
by using truncation arguments like in \cite{MR2743029}. 
\end{Remark}

\begin{Theorem}\label{thm:large_deviation_orthomartingale_stationary_p_geq_2}
 Let $B$ be a separable $2$-smooth Banach space and $p>2$. For each $B$-valued orthomartingale 
 differences random field $\pr{X_{\gr{i}}}_{\gr{i}\in \Z^d}$ 
 satisfying \eqref{eq:stationarity_of_sums} and 
 such that $X_{\gr{1}}\in \mathbb L^p\log^{d-1}\mathbb L$, for each positive $\varepsilon$ and 
 each $\alpha\in \left(1/2,1\right]$, 
\begin{equation}\label{eq:Baum_Katz_champs_pgeq_2}
\sum_{\gr{n}\in\N^d}\abs{\gr{n}}^{p\pr{\alpha-1/2}-1}\PP\ens{\max_{\gr{1}\imd \gr{i}\imd\gr{n}}
\norm{S_{\gr{i}}}>\varepsilon \abs{\gr{n}}^{\alpha}}<+\infty.
\end{equation}
\end{Theorem}

\begin{Remark}
One hand the results in \cite{MR3451971}, we do not require boundedness of 
the conditional moments. On the other hand, their result do not require 
that $\pr{\abs{X_{\gr{i}}}}_{\gr{i}\in\Z^d}$ is identically distributed hence the 
results are not directly comparable.
\end{Remark}

\section{Proofs}
\label{sec:proofs}

\subsection{Proofs of Theorems~\ref{thm:deviation_inequality_non_stationary_martingale_diff} 
and \ref{thm:deviation_inequality_stationary_martingale_diff}}

\begin{proof}[Proof of Theorem~\ref{thm:deviation_inequality_non_stationary_martingale_diff}]
 We first start by a distribution function inequality, which was first established in the 
 real valued case and $r'=2$ in \cite{MR0365692} (see
 also \cite{MR0394135}, p.~24 for a proof). 
 
 \begin{Lemma}\label{lem:functional_inequality}
  Let $\pr{B,\norm{\cdot}}$ be an $r$-smooth Banach space for some $1<r\leq 2$ and let 
  $1<r'\leq r$. Then for any $\delta\in (0,1)$ and any $B$-valued martingale differences 
  sequence $\pr{X_i}_{i\geq 1}$ with respect to the filtration $\pr{\F_i}_{i\geq 1}$, the 
  following inequality holds for any $n\geq 1$ and $x>0$:
  \begin{multline}\label{eq:inegalite_queue_intermediaire}
   \PP\ens{\max_{1\leq i\leq n}\norm{S_i}>2x}
   \leq C_{r',B}\pr{\frac{\delta}{1-\delta}}^{r'}\PP\ens{\max_{1\leq i\leq n}\norm{S_i}>x}\\
   + \PP\ens{\max_{1\leq i\leq n}\norm{X_i}>\delta x}+
    \PP\ens{\pr{\sum_{i=1}^n\E{\norm{X_i}^{r'}\mid \F_{i-1}}}^{1/r'}>\delta x},
  \end{multline}
where $C_{r',B}$ is defined by \eqref{eq:r'_smooth_martingale} and $S_n=\sum_{i=1}^nX_i$.

 \end{Lemma}
\begin{proof}
We assume that $n\geq 2$ since for $n=1$, the result is obvious.
We define $A_1=B_1=C_1=\emptyset$ and for $2\leq i\leq n$, 
\begin{equation}
 A_i:=\ens{\max_{1\leq u\leq i-1}\norm{S_u}\in (x,2x)},
\end{equation}
\begin{equation}
 B_i:=\ens{\max_{1\leq u\leq i-1}\norm{X_u}\leq \delta x},
\end{equation}
\begin{equation}
 C_i:=\ens{\sum_{u=1}^i  \E{\norm{X_i}^{r'} \mid \F_{u-1}}\leq \pr{\delta x}^{r'}}.
\end{equation}
We then introduce 
\begin{equation}
 Y_i:=\mathbf{1}_{A_i}\mathbf{1}_{B_i}\mathbf{1}_{C_i}X_i,\quad 1\leq i\leq n.
\end{equation}
We show that the following inclusion holds:
\begin{multline}\label{eq:inclusion_cle}
\ens{ \max_{1\leq i\leq n}\norm{S_i}>2x}\cap \ens{\max_{1\leq i\leq n}\norm{X_i}\leq \delta x}
 \cap \ens{\pr{\sum_{i=1}^n\E{\norm{X_i}^{r'}\mid \F_{i-1}}}^{1/r'}\leq \delta x}\\
 \subset \ens{\norm{\sum_{i=1}^nY_i}> \pr{1-\delta}x  }.
\end{multline}
Indeed, let $\omega$ be an element of the left hand side of \eqref{eq:inclusion_cle}. Then 
for any $i\in\ens{2,\dots,n}$, $\omega$ belongs to $B_i\cap C_i$. Consequently, 
$\sum_{i=1}^nY_i\pr{\omega}= \sum_{i=2}^n\mathbf{1}_{A_i}X_i\pr{\omega}$. Let $I:=
\ens{i\in\ens{2,\dots,n}: \omega\in A_i}$. Let 
$M_i:=\max_{1\leq u\leq i}\norm{S_i}$. Note that 
$\norm{S_1\pr{\omega}}=\norm{X_1\pr{\omega}}\leq \delta x<x$, since $0<\delta<1$ 
hence $M_1<x$ and $M_n>2x$. Since for any $i\in\ens{1,\dots,n-1}$ we have $\norm{X_i\pr{\omega}}
\leq \delta x$, it follows that
$0\leq M_{i+1}-M_i\leq \delta x$. Consequently, $I$ is of the form $\ens{i, i_0\leq i\leq j_0}$ 
for some integers $i_0\geq 2$ and $j_0\leq n$. Therefore, 
\begin{equation}
 \norm{\sum_{i=1}^nY_i\pr{\omega}}= \norm{\sum_{i=i_0}^{j_0}X_i\pr{\omega}}
 \geq \norm{\sum_{i=1}^{j_0}X_i\pr{\omega}}-\norm{\sum_{i=1}^{i_0-1}X_i\pr{\omega}}
 -\norm{X_{i_0}\pr{\omega}}.
\end{equation}
Now, \eqref{eq:inclusion_cle} holds in view of the inequalities $
\norm{\sum_{i=1}^{j_0}X_i\pr{\omega}}>2x$, $\norm{\sum_{i=1}^{i_0-1}X_i\pr{\omega}}
\leq x$ (since $\omega\in I_{i_0}$) and $\norm{X_{i_0}\pr{\omega}}\leq \delta x$.

Taking the probabilities on both sides in \eqref{eq:inclusion_cle}, one gets 
\begin{multline}\label{eq:inclusion_cle_bis}
\PP\pr{\ens{ \max_{1\leq i\leq n}\norm{S_i}>2x}\cap \ens{\max_{1\leq i\leq n}\norm{X_i}\leq \delta x}
 \cap \ens{\pr{\sum_{i=1}^n\E{\norm{X_i}^{r'}\mid \F_{i-1}}}^{1/r'}\leq \delta x}}\\
 \leq \PP\ens{\norm{\sum_{i=1}^nY_i}^{r'}> \pr{\pr{1-\delta}x}^{r'}  }
 \leq  \pr{\pr{1-\delta}x}^{r'}  \E{\norm{\sum_{i=1}^nY_i}^{r'}}.
\end{multline}
Observe that $A_i$, $B_i$ and $C_i$ belong to $\F_{i-1}$, hence $\pr{Y_i}_{i\geq 1}$ is a martingale 
differences sequence. The combination of \eqref{eq:inclusion_cle_bis} with 
\eqref{eq:r'_smooth_martingale} yields 
\begin{multline}\label{eq:inclusion_cle_ter}
\PP\pr{\ens{ \max_{1\leq i\leq n}\norm{S_i}>2x}\cap \ens{\max_{1\leq i\leq n}\norm{X_i}\leq \delta x}
 \cap \ens{\pr{\sum_{i=1}^n\E{\norm{X_i}^{r'}\mid \F_{i-1}}}^{1/r'}\leq \delta x}}\\
 \leq  \pr{\pr{1-\delta}x}^{-r'} C_{r',B}\sum_{i=1}^n\E{\norm{Y_i}^{r'}}.
\end{multline}
Since for $i\geq 2$, 
\begin{equation}
 \E{\norm{Y_i}^{r'}}=\E{\E{\norm{Y_i}^{r'}}\mid \F_{i-1}}
 =\E{\mathbf{1}_{A_i}\mathbf{1}_{B_i}\mathbf{1}_{C_i}\E{\norm{X_i}^{r'}\mid \F_{i-1}}},
\end{equation}
we derive that 
\begin{equation}\label{eq:inclusion_cle_4}
 \E{\norm{Y_i}^{r'}}\leq \E{\mathbf{1}\ens{\max_{1\leq i\leq n}\norm{S_i}>x    }
 \mathbf{1}_{C_i}\E{\norm{X_i}^{r'}\mid \F_{i-1}}}.
\end{equation}
Observe that 
\begin{equation}\label{eq:inegalite_variances_cond}
 \sum_{i=2}^n\mathbf{1}_{C_i}\E{\norm{X_i}^{r'}\mid \F_{i-1}}\leq \pr{\delta x}^{r'}.
\end{equation}

Indeed, if $Z_i$ are non-negative random variables, $Z'_i:=\sum_{u=1}^iZ_u$
and $E_i=\ens{Z'_u\leq t}$, we have
 
\begin{align*}
 \sum_{i=2}^nZ_i\mathbf 1_{E_i}&=\sum_{i=2}^n\pr{Z'_i-Z'_{i-1}}\mathbf 1_{E_i}\\
 &=\sum_{j=2}^n Z'_j\mathbf 1_{E_j}-\sum_{j=1}^{n-1}Z'_j\mathbf 1_{E_{j+1}}\\
 &=Z'_n\mathbf 1_{E_n}+\sum_{j=2}^{n-1}Z'_j\pr{\mathbf 1_{E_j}-\mathbf 1_{E_{j+1}}}-Z'_1
 \mathbf 1_{E_2},
\end{align*}
since $E_{j+1}\subset E_j$, 
the second term is smaller than 
$\sum_{j=2}^{n-1}t\pr{\mathbf 1_{E_j}-\mathbf 1_{E_{j+1}}}=t \mathbf 1_{E_2}-
t \mathbf 1_{E_n}$ and consequently, 
\begin{equation}
 \sum_{i=2}^nZ_i\mathbf 1_{E_i}\leq \pr{Z'_n-t}\mathbf 1_{E_n}+t\mathbf 1_{E_2}-Z'_1
 \mathbf 1_{E_2}\leq t.
\end{equation}
Combining \eqref{eq:inclusion_cle_ter}, \eqref{eq:inclusion_cle_4} and 
\eqref{eq:inegalite_variances_cond}, we get \eqref{eq:inegalite_queue_intermediaire}. This ends 
the proof of Lemma~\ref{lem:functional_inequality}.
\end{proof}
 Let us define the functions 
 \begin{equation}
  f\colon x\mapsto \PP\ens{\max_{1\leq i\leq n}\norm{S_i}>x} \mbox{ and }
  \end{equation}
  \begin{equation}
  g\colon x\mapsto \PP\ens{\max_{1\leq i\leq n}\norm{X_i}> x}+
    \PP\ens{\pr{\sum_{i=1}^n\E{\norm{X_i}^{r'}\mid \F_{i-1}}}^{1/r'}> x}.
 \end{equation}
We established in Lemma~\ref{lem:functional_inequality} that for any $x>0$ and any 
$\delta\in (0,1)$, 
\begin{equation}
 f\pr{2x}\leq C_{r',B}\pr{\frac{\delta}{1-\delta}}^{r'}f\pr{x}+g\pr{\delta x}.
\end{equation}
Let $q>0$ be fixed and $\eta:= C_{r',B}\pr{\frac{\delta}{1-\delta}}^{r'}$. Let $t>0$ be fixed, 
$a_n:=f\pr{2^nt}$, $b_n:=\eta^{-n}a_n$ and $c_n:=g\pr{2^nt\delta}$. Then 
\begin{equation}
 b_{n+1}=\eta^{-n-1}a_{n+1}\leq \eta^{-n-1}\pr{\eta a_n+c_n}=
 b_n+\eta^{-n-1}c_n.
\end{equation}
Consequently, 
\begin{equation}
 b_N=b_0+\sum_{n=0}^{N-1}b_{n+1}-b_n\leq a_0+\sum_{n=0}^{N-1}\eta^{-n-1}c_n,
\end{equation}
which gives 
\begin{equation}
 a_N\leq a_0\eta^N+\sum_{n=0}^{N-1}\eta^{N-n-1}c_n,
\end{equation}
and with the change of index $j=N-n$, we derive that for any positive $t$ and any 
integer $N$, 
\begin{equation}\label{eq:functional_inequality_etape_intermediaire}
 f\pr{2^Nt}\leq  f\pr{t}\eta^N+\sum_{j=1}^{N}\eta^{j-1}g\pr{\delta 2^{N-j}t}.
\end{equation}
Now, we choose $\delta:=2^{-1-q/r'}C_{r',B}^{-1/r'}$, which is smaller than $1$, 
as $C_{r',B}$ is bigger than $1$.
Applying \eqref{eq:functional_inequality_etape_intermediaire} with $x=2^Nt$ and letting $N$ going to infinity (accounting 
$f\pr{2^{-N}x}\leq 1$ and $0<\eta<1$), we get 
\begin{equation}
 f\pr{x}\leq  \sum_{j=1}^{+\infty}\eta^{j-1}g\pr{\delta 2^{-j}x}.
\end{equation}
Since the function $g$ is non-increasing, we have 
\begin{equation}
 \int_{2^{-j}}^{2^{-j+1}}g\pr{ux\delta}u^{q-1}\mathrm du
 \geq g\pr{2^{-j}x\delta}\int_{2^{-j}}^{2^{-j+1}}u^{q-1}\mathrm du
 =g\pr{2^{-j}x\delta}\frac{2^q-1}q2^{-jq}
\end{equation}
hence 
\begin{equation}
 f\pr{x}\leq  \frac{q}{2^q-1}\sum_{j=1}^{+\infty}\eta^{j-1}2^{jq}
 \int_{2^{-j}}^{2^{-j+1}}g\pr{ux\delta}u^{q-1}\mathrm du.
\end{equation}
Notice that
\begin{equation}
 \eta\leq C_{r',B}2^{r'}2^{-r'-q}C_{r',B}^{-1}\leq 2^{-q},
\end{equation}
hence 
\begin{equation}
 f\pr{x}\leq  \frac{q}{2^q-1}\eta^{-1} 
 \int_{0}^{1}g\pr{ux\delta}u^{q-1}\mathrm du.
\end{equation}
Since 
\begin{equation}
 \eta^{-1}=\pr{\frac{1-\delta}{\delta}}^{r'}C_{r',B}^{-1}
 \leq \pr{\frac{1}{\delta}}^{r'}C_{r',B}^{-1}\leq 2^{q-r'},
\end{equation}
we get \eqref{eq:deviation_inequality_non_stationary_martingale_diff}. This ends the proof 
of Theorem~\ref{thm:deviation_inequality_non_stationary_martingale_diff}.
\end{proof}

\begin{proof}[Proof of Theorem~\ref{thm:deviation_inequality_stationary_martingale_diff}]

We shall need the following lemma. 
\begin{Lemma}\label{Lemma_weak_type_estimate}
  Assume that $X$ and $Y$ are two 
  non-negative random variables such that for each positive $x$, 
  we have 
  \begin{equation}\label{eq:weak_type_assumption}
   x\PP\ens{X>x}\leqslant\mathbb 
  E\left[Y \mathbf 1\ens{X\geqslant x}\right].
  \end{equation}

  Then for each $t$, the following inequality holds:
  \begin{equation}
   \PP\ens{X>2t}\leqslant \int_1^{+\infty}\PP\ens{Y>st}\mathrm ds.
  \end{equation}

 \end{Lemma}
  \begin{proof}[Proof of Lemma~\ref{Lemma_weak_type_estimate}]
  Rewriting the expectation as
  \begin{equation}\label{weak_type_estimate}
   \mathbb E\left[Y \mathbf 1\ens{X\geqslant 2t}\right]=
   \int_0^{+\infty} \PP\ens{Y \mathbf 1\ens{X\geqslant 2t}>u}\mathrm du
   \leqslant t\PP\ens{X\geqslant 2t}+\int_t^{+\infty}\PP\ens{Y>u}\mathrm du,
  \end{equation}
 we derive by the assumption the bound 
 \begin{equation}
  2t\PP\ens{X>2t} \leqslant t\PP\ens{X\geqslant 2t}+\int_t^{+\infty}\PP\ens{Y>u}\mathrm du.
 \end{equation}
  We conclude using the substitution $ts:=u$.
 \end{proof}

We apply Theorem~\ref{thm:deviation_inequality_non_stationary_martingale_diff}.
The first term of \eqref{eq:deviation_inequality_non_stationary_martingale_diff} 
is controlled in the following way, using the fact that if $U$ has uniform distribution 
on $[0,1]$, then $Q_{\norm{X_i}}\pr{U}$ has the same distribution as $\norm{X_i}$:
\begin{align}
 \int_0^1\PP\ens{\max_{1\leqslant i\leqslant n}
    \norm{X_i}>xun^{1/r'}}u^{q-1}\mathrm du&
 \leq \sum_{i=1}^n \int_0^1\PP\ens{ 
    \norm{X_i}> xun^{1/r'}}u^{q-1}\mathrm du\\
 &=\sum_{i=1}^n \int_0^1\lambda\ens{t\in [0,1],
    Q_{\norm{X_i}}(t)> xun^{1/r'}}u^{q-1}\mathrm du\nonumber\\
 &\leq \sum_{i=1}^n \int_0^1\lambda\ens{t\in [0,1],
    Q_{X}(t)>B_{r',q}xun^{1/r'}}u^{q-1}\mathrm du\nonumber\\
    &=n\int_0^1\PP\ens{X> xun^{1/r'}}u^{q-1}\mathrm du,
\end{align}
where $\lambda$ denotes the Lebesgue measure.

In order to control the second term of \eqref{eq:deviation_inequality_non_stationary_martingale_diff}, we first 
bound $\sum_{i=1}^n
    \E{\norm{X_i}^{r'}\mid \F_{i-1}} $ by $\sum_{i=1}^nV_i$ and 
we notice that for any convex function $\phi\colon\R\to \R$, 
\begin{equation}
 \E{\phi\pr{\frac 1n\sum_{i=1}^n V_i}}
 \leq \frac 1n\sum_{i=1}^n
    \E{\phi\pr{V_i}}=\E{\phi\pr{V_1}}.
\end{equation}

 By Theorem~6 in \cite{MR606989}, there exists a probability space $\pr{\Omega',\mathcal A',\PP'}$ 
 and random variables $Z'_n$ and $Z'$ such that $V'_n$ has the same distribution 
 as $\frac 1n\sum_{i=1}^nV_i$, $Z'$ has the same distribution as $V_1$ and 
 such that $Z'_n=\E{Z'\mid Z'_n}$.
 
Therefore, inequality \eqref{weak_type_estimate} holds 
with $\displaystyle X:=Z'_n$ and
$Y=Z'$,
 hence by Lemma~\ref{Lemma_weak_type_estimate} the estimate 
\begin{equation}\label{eq:maximal_ergodic}
 \PP\ens{\frac 1n\sum_{i=1}^n
    \E{\norm{X_i}^{r'}\mid \F_{i-1} }>2u^{r'}x^{r'}}\leqslant 
 \int_1^{+\infty}\PP\ens{V_1>u^{r'}x^{r'}s}
  \mathrm ds
\end{equation}
 is valid for any $n$. We can deduce from 
 inequalities~\eqref{eq:deviation_inequality_non_stationary_martingale_diff} and
  \eqref{eq:maximal_ergodic} that \eqref{eq:deviation_inequality_stationary_martingale_diff_cond_var} is 
 satisfied after having used the elementary identity 
 \begin{equation}
  \int_0^1\int_1^{+\infty}h\pr{u^{r'}v}\mathrm dvu^{q-1}\mathrm du
  =\frac 1{q-r'}\int_0^{+\infty}h\pr{w}\min\ens{w^{\frac{q-r'}{r'}},1}\mathrm dw
 \end{equation}
with $h\pr{t}:=\PP\ens{V_1  >x^{r'}t/2}$.
 
 In order to prove \eqref{eq:deviation_inequality_stationary_martingale_diff}, we bound the two terms 
of the right hand side of \eqref{eq:deviation_inequality_non_stationary_martingale_diff}
 independently of $n$. 
Let us start by the first term, which can be written as 
\begin{equation}
\displaystyle n^{-q/r'} \int_0^{n^{1/r'}}\PP\ens{\max_{1\leqslant i\leqslant n}
    \norm{X_i}>2^{-1-q/r'}C_{r',B}^{-1/r'}xv}v^{q-1}\mathrm dv. 
\end{equation}
If $v\leqslant 1$, 
we use the bound $n^{1-q/r'}v^{q-1}\leqslant v^{q-1}$ (since $q>r'$). 
If $1<v\leqslant n^{1/r'}$, then  $n^{1-q/r'}v^{q-1}
\leqslant v^{r'-1}$. We thus have 
\begin{multline}\label{eq:bound_term_with_maximums}
 \int_0^1\PP\ens{\max_{1\leqslant i\leqslant n}
    \norm{X_i}>2^{-1-q/r'}C_{r',B}^{-1/r'}n^{1/r'}xu}u^{q-1}\mathrm du\\
 \leq 
 \int_0^{+ \infty} \PP\ens{\norm{X_1}> 2^{-1-q/r'}C_{r',B}^{-1/r'}xv}\min\ens{v^{q-1},v^{r'-1}}\mathrm dv.
 \end{multline}
 
Let us treat the second term. For any convex function $\phi\colon\R\to \R$, 
\begin{align*}
 \E{\phi\pr{\frac 1n\sum_{i=1}^n \E{\norm{X_i}^{r'}\mid \F_{i-1}  }   }}
 &\leq \frac 1n\sum_{i=1}^n
    \E{\phi\pr{\E{\norm{X_i}^{r'}\mid \F_{i-1}  }}} \\
    & \leq \frac 1n\sum_{i=1}^n
    \E{\E{\phi\pr{\norm{X_i}^{r'} }\mid \F_{i-1} }}=\E{\phi\pr{\norm{X_1}^{r'}  }}.
\end{align*}
Using again Theorem~6 in \cite{MR606989}, we derive that 
\begin{multline}\label{eq:control_conditional_variances_identically_distributed}
 \int_0^1\PP\ens{
    \pr{\sum_{i=1}^n
    \E{\norm{X_i}^{r'}\mid \F_{i-1}}}^{1/r'}>2^{-1-q/r'}C_{r',B}^{-1/r'}xu}u^{q-1}\mathrm du
 \\
 \leq     \int_0^1\int_1^{+\infty}\PP\ens{
    \norm{X_1} >2^{-1-q/r'}C_{r',B}^{-1/r'}xu}u^{q-1}\mathrm du.
\end{multline}
 
 Combining 
 \eqref{eq:bound_term_with_maximums} and \eqref{eq:control_conditional_variances_identically_distributed}, we
  get \eqref{eq:deviation_inequality_stationary_martingale_diff}. This 
 ends the proof of Theorem~\ref{thm:deviation_inequality_stationary_martingale_diff}.
\end{proof}

\subsection{Proof of Theorem~\ref{thm:deviation_inequality_stationary_orthomartingale_diff}}

 Let us prove \eqref{eq:deviation_inequality_stationary_orthomartingale_diff}. Let $B$ be a 
 separable $r$-smooth Banach space and let $r'\in (1,2]$, $q>r'$ be fixed.
 
 \begin{Lemma}\label{lem:Lemma_proof_inequality_orthomatingales}
Let $\left(f_d\right)_{d\geq 1}$ be a sequence of functions from $\left(0,+\infty\right)$ to ifself 
such that:
\begin{enumerate}
\item\label{itm:first_property_function_fd} for any martingale differences sequence $\left(X_i\right)_{i\geq 1}$ with values in 
$B$ such that $\left(\norm{X_i}\right)_{i\geq 1}$ is identically distributed
 and $\E{\norm{X_1}^{r'}}<+\infty$, any $n\geq 1$
and any positive $x$,
\begin{equation}
\PP\ens{\max_{1\leq i\leq n}\norm{S_i}>n^{1/r'}x}\leq \int_{0}^{+\infty}
\PP\ens{\norm{X_1}>xv   }f_1\pr{v}\mathrm dv,
\end{equation}
where $S_n=\sum_{i=1}^nX_i$;
\item\label{itm:second_property_function_fd} for any $d\geq 2$ and any positive $w$,
\begin{equation}\label{eq:recurrence_fd}
f_d\pr{w}\geq\int_0^{+\infty}\int_0^{+\infty}f_{d-1}\pr{u}f_{1}\pr{u'}\frac{1}{uu'}
\mathbf{1}\ens{uu'\leq w}\mathrm du\mathrm du'.
\end{equation}
\end{enumerate}
Then for any integer $d\geq 1$, any orthomartingale differences random field
$\pr{X_{\gr{i}}}_{\gr{i}\in \Z^d}$ satisfying \eqref{eq:stationarity_of_sums}, 
$\E{\norm{X_\gr{1}}^{r'}}<+\infty$, any $\gr{n}\smd \gr{1}$ and any positive $x$, 
\begin{equation}\label{eq:inegalite_dans_HDR}
\PP\ens{\max_{\gr{1}\imd \gr{i}\imd \gr{n}} 
\norm{S_{\gr{i}}}> 2^{d-1}x\abs{\gr{n}}^{1/r'}}
\leq \int_{0}^{+\infty}
\PP\ens{\norm{X_{\gr{1}}}>xv   }f_d\pr{v}\mathrm dv,
\end{equation}
where $S_{\gr{i}}$ is defined by \eqref{eq:definitio_of_partial_sums}.
 \end{Lemma}
 
 For $p>0$ and $k\in \N$, let 
 \begin{equation}
 a_{p,k}:=\int_0^1t^{p-1}\pr{1+\abs{\log t}}^{k}\mathrm{dt}.
 \end{equation}
\begin{Lemma}\label{lem:Lemma2_proof_inequality_orthomatingales}
Let $\pr{c_d}_{d\geqslant 1}$ be the sequence of real numbers such that 
\begin{equation}
c_1=\frac{2^{q+1}}{2^q-1}\frac{q2^{-r'}}{q-r'}\mbox{ and }
\end{equation} 
\begin{equation}
 c_{d}=c_{d-1}\pr{1+a_{q-r',d-2}
+2^{d-1}a_{q-r',d-2}}.
 \end{equation}
 Then the sequence of functions $\left(f_d\right)_{d\geq 1}$ defined by  
 \begin{equation}
 f_d\colon u\mapsto  c_d \min\ens{u^{q-1},u^{r'-1}}
 \pr{1+\abs{\log u}}^{d-1}, \quad u>0
 \end{equation}
 satisfies the conditions of Lemma~\ref{lem:Lemma_proof_inequality_orthomatingales}.
\end{Lemma}

Inequality~\eqref{eq:deviation_inequality_stationary_orthomartingale_diff} is a
direct consequence of Lemmas~\ref{lem:Lemma_proof_inequality_orthomatingales} and 
\ref{lem:Lemma2_proof_inequality_orthomatingales}. 

\begin{proof}[Proof of Lemma~\ref{lem:Lemma_proof_inequality_orthomatingales}]
The proof is done by induction on $d$. The case $d=1$ is contained in the assumptions. 
Assume that inequality \eqref{eq:inegalite_dans_HDR} holds for some $d\geq 1$ for any 
orthomartingale differences random field $\pr{X_{\gr{i}}}_{\gr{i}\in \Z^d}$ satisfying \eqref{eq:stationarity_of_sums}, 
$\E{\norm{X_\gr{1}}^{r'}}<+\infty$, any $\gr{n}\smd \gr{1}$ and any positive $x$. 

Let  $\pr{X_{\gr{i}}}_{\gr{i}\in \Z^{d+1}}$ be an orthomartingale 
differences random field with respect to the commutatitve filtration 
$\pr{\F_{\gr{i}}}_{\gr{i}\in \Z^{d+1}}$ satisfying \eqref{eq:stationarity_of_sums}. 
Using property~\ref{submartingale} in Lemma~\ref{lem:properties_orthomartingales} 
and Lemma~\ref{Lemma_weak_type_estimate} applied to 
\begin{equation}
X=\max_{\substack{1\leq i_k\leq n_k\\ 1\leq k\leq d+1}   } 
\norm{S_{\gr{i}}}\mbox{ and }Y=\max_{\substack{1\leq i_k\leq n_k\\ 1\leq k\leq d}   } 
\norm{S_{\gr{i},n_{d+1}}},
\end{equation}
we get 
\begin{equation}
\PP\ens{\max_{\gr{1}\imd \gr{i}\imd \gr{n}} 
\norm{S_{\gr{i}}}> 2^dx\abs{\gr{n}}^{1/r'}}
\leq \int_{1}^{+\infty}
\PP\ens{\max_{\substack{1\leq i_k\leq n_k\\ 1\leq k\leq d}   } 
\norm{S_{\gr{i},n_{d+1}}} >2^{d-1}x\abs{\gr{n}}^{1/r'}v   }\mathrm dv.
\end{equation}
We now apply the induction hypothesis to  
$ \widetilde{X_{\gr{i}}} := \sum_{k=1}^{n_{d+1}} X_{\gr{i},k} $, 
$\widetilde{\F_{\gr{i}}}:=  \F_{\gr{i},n_{d+1}}$ 
and $\widetilde{x}:=xn_{d+1}^{1/r'}$ to get 
\begin{equation}
\PP\ens{\max_{\gr{1}\imd \gr{i}\imd \gr{n}} 
\norm{S_{\gr{i}}}> 2^dx\abs{\gr{n}}^{1/r'}}
\leq \int_{1}^{+\infty}\int_0^{+\infty}
\PP\ens{ \norm{\sum_{k=1}^{n_{d+1}} X_{\gr{1},k}}
 >xn_{d+1}^{1/r'}vu   }f_{d-1}\pr{u}\mathrm du\mathrm dv.
\end{equation}
After having appylied the one dimensional case, we derive that 
\begin{equation*}
\PP\ens{\max_{\gr{1}\imd \gr{i}\imd \gr{n}} 
\norm{S_{\gr{i}}}> 2^dx\abs{\gr{n}}^{1/r'}}\\
\leq \int_{\pr{0,+\infty}^3} 
\PP\ens{ \norm{ X_{\gr{1}}}
 >x vu u'  }f_{d-1}\pr{u}f_{1}\pr{u'}\mathbf 1\ens{v>1}\mathrm dv\mathrm du\mathrm du'.
\end{equation*}
and the substitution $w:=vu u'$ for fixed $u$ and $u'$ combined with \eqref{eq:recurrence_fd}
end the proof of Lemma~\ref{lem:Lemma_proof_inequality_orthomatingales}.
\end{proof}

\begin{proof}[Proof of Lemma~\ref{lem:Lemma2_proof_inequality_orthomatingales}]
Item~\ref{itm:first_property_function_fd} follows from Theorem~\ref{thm:deviation_inequality_stationary_martingale_diff} after a substitution in the integral of the 
right hand side of \eqref{eq:deviation_inequality_stationary_martingale_diff}. 

Let us show item~\ref{itm:second_property_function_fd}. Let $d\geq 2$ be fixed. Observe that 
\begin{align*}
\int_0^{+\infty}f_{1}\pr{u'}\frac{1}{u'}\mathbf{1}\ens{uu'\leq w}\mathrm du'
&=c_1\int_0^{w/u}\frac 1v \min\ens{v^{q-1},v^{r'-1}}\mathrm{d}v\\
&\leq c_1\min\ens{\int_0^{w/u}  v^{q-2}\mathrm{d}v,\int_0^{w/u} v^{r'-2}
 \mathrm{d}v}\\
 &\leq \frac{c_1}{r'-1}\min\ens{\pr{\frac{w}u}^{q-1},\pr{\frac{w}u}^{r'-1}},
\end{align*}
hence 
\begin{multline}
\int_0^{+\infty}\int_0^{+\infty}f_{d-1}\pr{u}f_{1}\pr{u'}\frac{1}{uu'}
\mathbf{1}\ens{uu'\leq w}\mathrm du\mathrm du'\\
\leq \frac{c_1}{r'-1}\int_0^{+\infty}f_{d-1}\pr{u}
\min\ens{\pr{\frac{w}u}^{q-1},\pr{\frac{w}u}^{r'-1}}\mathrm{d}u.
\end{multline}
Let $g\colon u\mapsto f_{d-1}\pr{u}u^{-1}
\min\ens{\pr{\frac{w}u}^{q-1},\pr{\frac{w}u}^{r'-1}}$ and 
$
I\pr{w}:=\int_0^{+\infty}g\pr{u}\mathrm{d}u$.
Assume that $w\leq 1$. Spliting the integral into three parts (from $0$ to $w$, from $w$ to $1$ and from $1$ to infinity), we get 
\begin{multline}\label{eq:decomposition_integrale_w_leq_1}
I\pr{w}=c_{d-1}\int_0^w u^{q-1}\pr{1+\abs{\log u}}^{d-2} u^{-1} \pr{\frac{w}u}^{r'-1}
\mathrm{d}u+
c_{d-1}\int_w^{1} u^{q-1}\pr{1+\abs{\log u}}^{d-2} u^{-1}\pr{\frac{w}u}^{q-1}\mathrm{d}u
\\+
c_{d-1}\int_1^{+\infty}u^{r'-1}\pr{1+\abs{\log u}}^{d-2}u^{-1} \pr{\frac{w}u}^{q-1}     \mathrm{d}u
=:c_{d-1}\pr{I_1\pr{w}+I_2\pr{w}+I_3\pr{w}  }.
\end{multline}
Let us bound these integrals. We have 
\begin{align*}
I_1\pr{w}=w^{r'-1}\int_0^w u^{q-r'-1}\pr{1+\abs{\log u}}^{d-2} 
\mathrm{d}u
\end{align*}
and the substitution $x=u/w$ gives 
\begin{equation}
I_1\pr{w}=w^{q-1}\int_0^1 x^{q-r'-1}\pr{1+\abs{\log x}+\abs{\log w} }^{d-2} 
\mathrm{d}x\leq w^{q-1}2^{d-1}a_{q-r',d-2}\pr{1+\abs{\log w}}^{d-2}.
\end{equation}
Observe that for $u\in (w,1)$, 
\begin{equation}
 u^{q-1}\pr{1+\abs{\log u}}^{d-2} u^{-1}\pr{\frac{w}u}^{q-1}
 =w^{q-1}\pr{1+\abs{\log u}}^{d-2} u^{-1}\leq w^{q-1} \pr{1+\abs{\log w}}^{d-2} u^{-1}
\end{equation}
hence 
\begin{equation}
I_2\pr{w}\leq w^{q-1} \pr{1+\abs{\log w}}^{d-1}.
\end{equation}
Finally, 
\begin{equation}
I_3\pr{w}=w^{q-1}\int_1^{+\infty}\frac{1}{u^{q-r'+1}}\pr{1+\abs{\log u}}^{d-2}
=a_{q-r',d-2}w^{q-1}
\end{equation}
hence 
\begin{equation}
I\pr{w}\leq c_{d-1}w^{q-1} \pr{1+\abs{\log w}}^{d-1}\pr{1+a_{q-r',d-2}
+2^{d-1}a_{q-r',d-2}}.
\end{equation}
Now, if $w>1$, a similar result by spliting the integral into three parts (from $0$ to $1w$, from $1$ to $w$ and from $1$ to infinity) yields for $w>1$:
\begin{equation}
I\pr{w}\leq c_{d-1}w^{r'-1} \pr{1+\abs{\log w}}^{d-1}\pr{1+a_{q-r',d-2}
+2^{d-1}a_{q-r',d-2}}.
\end{equation}
This concludes the proof of Lemma~\ref{lem:Lemma2_proof_inequality_orthomatingales}.
\end{proof}

 \subsection{Proof of the results of Section~\ref{sec:applications}}
 
\begin{proof}[Proof of Theorem~\ref{thm:linear_regression}]
 A computation gives that 
 \begin{equation}
  \theta_n-\theta=\frac{\sum_{i=1}^n\phi_i\eps_i}{\sum_{j=1}^n\phi_j^2}.
 \end{equation}
We define 
\begin{equation}
 \xi_i:=\frac{\phi_i\eps_i}{\sum_{j=1}^n\phi_j^2}
\end{equation}
\begin{equation}
 \F_i:=\sigma\pr{\eps_u, 1\leq u\leq i, \phi_j, 1\leq j\leq n},\quad i\geq 1, \F_0=
 \sigma\pr{\phi_j, 1\leq j\leq n},
\end{equation}
and $\mathcal G_i:=\sigma\pr{\eps_u, 1\leq u\leq i}$ for $i\geq 1$ and 
$\mathcal G_0=\ens{\emptyset,\Omega}$.
In this way, for $i\geq 2$, 
\begin{equation} 
 \E{\xi_i\mid \F_{i-1}}=
 \frac{\phi_i}{\sum_{j=1}^n\phi_j^2}\E{\eps_i\mid \F_{i-1}}.
\end{equation}
Since $\sigma\pr{ \phi_j, 1\leq j\leq n}$ is independent of 
$\sigma\pr{\eps_u, 1\leq u\leq i}$, equality 
\begin{equation}
\E{\eps_i\mid \F_{i-1}}=\E{\eps_i\mid \sigma\pr{\eps_u,1\leq u\leq i}}
\end{equation}
holds and the right hand side was assumed to be equal to zero. Moreover, by independence, 
$\E{\xi_1\mid \F_0}=0$ hence $\pr{\xi_i,\F_i}_{i\geq 1}$ is a martingale 
differences sequence.
Since $\pr{\theta_n-\theta}\sqrt{\sum_{i=1}^n\phi_i^2}=\sum_{i=1}^n\xi_i$, an application of 
Theorem~\ref{thm:deviation_inequality_non_stationary_martingale_diff} with $B=\R$ and $r'=2$ yields 
\begin{equation}\label{eq:regression_proof_first_step}
 \PP\ens{\abs{\theta_n-\theta}\sqrt{\sum_{i=1}^n\phi_i^2}>x}\leq A_1+A_2,
\end{equation}
where 
\begin{equation}
 A_1=\frac{2^{q-2}}{2^q-1}q \int_0^1\PP\ens{\max_{1\leqslant i\leqslant n}
    \abs{\xi_i}>2^{-1-q/2} xu}u^{q-1}\mathrm du,
\end{equation}

\begin{equation}
 A_2=\frac{2^{q-2}}{2^q-1}q \int_0^1\PP\ens{
    \pr{\sum_{i=1}^n
    \E{\xi_i^2\mid \F_{i-1}}}^{1/2}>2^{-1-q/2} xu}u^{q-1}\mathrm du.
\end{equation}
We bound $A_1$ using Markov's inequality:
\begin{align}
 A_1&\leq \frac{2^{q-2}}{2^q-1}q \sum_{i=1}^n\int_0^1\PP\ens{ 
    \abs{\xi_i}>2^{-1-q/2} xu}u^{q-1}\mathrm du\\
  &\leq \frac{2^{q-2}}{2^q-1}\frac{q}{q-p} \sum_{i=1}^n 
    \E{\abs{\xi_i}^p} 2^{p+pq/2} x^{-p}.
\end{align}
Using independence and the convexity inequality $\sum_{i=1}^n\abs{\phi_i}^p
\leq \pr{\sum_{i=1}^n\phi_i^2}^{p/2}$ valid for $p\geq 2$, we get that 
$\E{\abs{\xi_i}^p}\leq C_1$ hence 
\begin{equation}\label{eq:regression_proof_second_step}
 A_1\leq C_1\frac{2^{q-2}}{2^q-1}\frac{q}{q-p}  2^{p+pq/2} x^{-p}.
\end{equation}
Now, in order to bound $A_2$, we notice that 
\begin{equation}
\E{\xi_i^2\mid \F_{i-1}}=
\frac{\phi_i^2}{\sum_{j=1}^n\phi_j^2}\E{\eps_i^2\mid \F_{i-1}},
\end{equation}
and since $\eps_i^2$ is independent of $ \pr{ \phi_j, 1\leq j\leq n}$, we derive that 
\begin{equation}
 \E{\xi_i^2\mid \F_{i-1}}=
\frac{\phi_i^2}{\sum_{j=1}^n\phi_j^2}\E{\eps_i^2\mid \mathcal G_{i-1}}
\leq C_2\frac{\phi_i^2}{\sum_{j=1}^n\phi_j^2}.
\end{equation}
Consequently, 
\begin{equation}
  \pr{\sum_{i=1}^n
    \E{\xi_i^2\mid \F_{i-1}}}^{1/2}\leq \sqrt{C_2},
\end{equation}
and 
\begin{equation}\label{eq:regression_proof_third_step}
 A_2\leq \frac{2^{q-2}}{2^q-1}q 2^{q+q^2/2}x^{-q}C_2^{q/2}.
\end{equation}
Theorem~\ref{thm:linear_regression}  follows from the combination of 
\eqref{eq:regression_proof_first_step}, \eqref{eq:regression_proof_second_step} and 
\eqref{eq:regression_proof_third_step}.

\end{proof}

 \begin{proof}[Proof of Theorem~\ref{thm:large_deviation_martingale_stationary_p_leq_2}]
 We use inequality \eqref{eq:deviation_inequality_stationary_martingale_diff} with $r'=r$ and $   q=2r$ to get 
 that for some constants $C$ and $c$ depending only on $r$ and $B$,  
 \begin{multline}
 \PP\ens{\max_{1\leq i\leq n}\norm{S_i}>n^{\alpha-1/r} n^{1/r}x}
 \leq Cn\int_0^1\PP\ens{X>cn^\alpha xu} u^{2r-1}\mathrm du
 \\+C\int_0^{+\infty}\PP\ens{V_1^{1/r}>cn^{\alpha-1/r}xu   }\min\ens{u^{2r-1},u^{r-1}}  \mathrm du.
 \end{multline}
 Observe that 
 \begin{align*}
 \sum_{n=1}^{+\infty}n^{r\alpha-1}\PP\ens{X>cn^\alpha xu}
 &= \sum_{n=1}^{+\infty}n^{r\alpha-1}\sum_{k= n}^{+\infty}\PP\ens{ 
X\in \left(k^\alpha xu,\pr{k+1}^\alpha xu\right] }\\
&=\sum_{k=1}^{+\infty}\sum_{n=1}^k n^{r\alpha-1}\PP\ens{ 
X\in \left(k^\alpha xu,\pr{k+1}^\alpha xu\right] }\\
&\leq \sum_{k=1}^{+\infty} k^{r\alpha}\PP\ens{ 
X\in \left(k^\alpha xu,\pr{k+1}^\alpha xu\right] }\\
&\leq \pr{xu}^{-r}\E{X^r}
 \end{align*}
 hence 
\begin{equation}
C \sum_{n=1}^{+\infty}n^{r\alpha-2}n\int_0^1\PP\ens{X>cn^\alpha xu} u^{2r-1}\mathrm du
\leq C  c^{-r}x^{-r}\int_0^1u^{r-1}\mathrm du.
\end{equation} 
Since for any non-negative random variable $Y$, $
 \sum_{n=1}^{+\infty}n^{r\alpha-2}\PP\ens{Y>n^{\alpha-1/r}   }
 \leq \E{Y\mathbf 1\ens{Y\geq 1}}$, we have
\begin{equation}
 \sum_{n=1}^{+\infty}n^{r\alpha-2}\PP\ens{V_1^{1/r}>cn^{\alpha-1/r}xu   }
 \leq \pr{cux}^{-r}\E{V_1\mathbf 1\ens{V_1^{1/r}\geq cxu}},
\end{equation}
which implies 
\begin{multline}
 \sum_{n=1}^{+\infty}n^{r\alpha-2}C\int_0^{+\infty}\PP\ens{V_1^{1/r}>cn^{\alpha-1/r}xu   }\min\ens{u^{2r-1},u^{r-1}}  \mathrm du\\
 \leq C\pr{cx}^{-r}\int_0^{+\infty}\E{V_1\mathbf 1\ens{V_1^{1/r}\geq cxu}}\min\ens{u^{r-1},u^{-1}}  \mathrm du.
\end{multline}
Now for any non-negative real number $y$, let $h\pr{y}:=\int_0^y\min\ens{u^{r-1},u^{-1}}  \mathrm du$.
If $y\leq 1$, then $h\pr{y}=y^r/r$ and if $y>1$, then 
\begin{equation}
h\pr{y}=\int_0^1u^{r-1}\mathrm du+
\int_1^{y}u^{-1}\mathrm du=\frac 1r+\log y.
\end{equation}
Since 
\begin{equation*}
\int_0^{+\infty}\E{V_1\mathbf 1\ens{V_1^{1/r}\geq cxu}}\min\ens{u^{r-1},u^{-1}}  \mathrm du
=\E{V_1 h\pr{ \frac{V_1^{1/r}}{cx}    } }\leq \frac 1r\E{V_1}+\frac 1r\E{V_1\log^+\pr{V_1}},
\end{equation*}
 we get the convergence of the series $
 \sum_{n=1}^{+\infty}n^{r\alpha -2}\PP\ens{\max_{1\leq i\leq n}\norm{S_i}>n^{\alpha}x}$.

 \end{proof}

 \begin{proof}[Proof of Theorem~\ref{thm:large_deviation_martingale_stationary_p_geq_2}]
  We use inequality \eqref{eq:deviation_inequality_stationary_martingale_diff_cond_var}
  with $\widetilde{x}:=x2^{n\pr{\alpha-1/2}}$,  $r'=2$ and $q=2p$. We get 
  \begin{multline}\label{eq:large_deviation_proof_1}
  n^{p\pr{\alpha -1/2}}\PP\ens{\max_{1\leq i\leq n}\norm{S_i}>n^{\alpha}x}\leq 
  C n^{p\pr{\alpha -1/2}+1} \int_0^1\PP\ens{
    X>cxun^{\alpha}}u^{2p-1}\mathrm du\\+C  n^{p\pr{\alpha -1/2}}\int_0^{+\infty}
    \PP\ens{V_1>x^2u^2 n^{2\alpha-1}}   \min\ens{u^{2p-1},u}\mathrm du.
  \end{multline}
  
\begin{enumerate}
 \item Assume that $X_1$ belongs to $\mathbb L^{p/2+1,\infty}$ and 
 $V_1\in \mathbb L^{p/2,\infty}$. One bounds the first term of the 
 right hand side of \eqref{eq:large_deviation_proof_1} by 
 \begin{multline}
  C n^{p\pr{\alpha -1/2}+1}\sup_{t>0}t^{p/2+1}\PP\ens{
    X>t}\int_0^1\prt{cxun^{\alpha}}^{-\prt{p/2+1}}u^{2p-1}\mathrm du\\
    =n^{   \pr{\frac p2-1}\pr{\alpha-1}} C\sup_{t>0}t^{p/2+1}\PP\ens{
    X>t}\prt{cx}^{-\prt{p/2+1}} \int_0^1u^{3p/2-2}\mathrm du
 \end{multline}
 and use $\pr{\frac p2-1}\pr{\alpha-1}\leq 0$.
One bounds the second term of the 
 right hand side of \eqref{eq:large_deviation_proof_1} by 
 \begin{multline}
  C n^{p\pr{\alpha -1/2}}\sup_{t>0}t^{p/2} 
  \PP\ens{V_1>t   }  \int_0^{+\infty}
 \prt{ cx^2u^2 n^{2\alpha-1} }^{-p/2}  \min\ens{u^{2p-1},u}\mathrm du\\
 = C\prt{cx}^{-p} \sup_{t>0}t^{p/2} 
  \PP\ens{V_1>t   }  \int_0^{+\infty}
  \min\ens{u^{p-1},u^{1-p}}\mathrm du,
 \end{multline}
and since $p>2$, the latter integral is finite. 

\item Assume that $X_1\in \mathbb L_0^{p/2+1,\infty}$ and 
 $V_1\in \mathbb L^{p/2,\infty}$. Plugging the bounds 
 \begin{equation}
  \PP\ens{X>cxun^{\alpha}} \leq 
  \prt{cxun^{\alpha}}^{-p/2-1}
    \sup_{t>cxu2^n}t^{p/2+1}\PP\ens{X>t}\mbox{ and }
 \end{equation}
\begin{equation}
 \PP\ens{V_1>cx^2u^2 n^{2\alpha-1}   }
 \leq \prt{cx^2u^2n^{2\alpha-1} }^{-p/2}\sup_{t>cx^2u^2 n^{2\alpha-1} }
    t^{p/2}\PP\ens{V_1>t   }.
\end{equation}
into \eqref{eq:large_deviation_proof_1}, we get 
\begin{multline}
 n^{p\pr{\alpha-1/2}} \PP\ens{\max_{1\leq i\leq n}\norm{S_i}>n^\alpha x} \leq 
  C\prt{cx}^{-p/2-1}\int_0^1 \sup_{t>cxun^\alpha}t^{p/2+1}\PP\ens{X>t}
  u^{3s/2-2}\mathrm du\\
  +C\prt{cx^2}^{-p/2} \int_0^{+\infty}
  \sup_{t>cx^2u^2 n^\alpha}
    t^{p/2}\PP\ens{V_1>t   }\max\ens{u^{p-1}, 
    u^{1-p}}\mathrm du,
\end{multline}
and the right hand side goes to zero by monotone convergence.
\item Assume that $X\in\mathbb L^{p/2+1}$ and $V_1
 \in \mathbb L^{p/2}$. In view of \eqref{eq:large_deviation_proof_1}, we have 
\begin{multline}
 \sum_{n=1}^{+\infty}n^{p\pr{\alpha-1/2}}
  \PP\ens{\max_{1\leq i\leq n}\norm{S_i\prt m}>n^{\alpha}x} \leq 
  C \sum_{n=1}^{+\infty}n^{p\pr{\alpha-1/2}+1} \int_0^1\PP\ens{
   X>cxun^{\alpha}}u^{2p-1}\mathrm du\\+C \sum_{n=1}^{+\infty}n^{p\pr{\alpha-1/2}}
   \int_0^{+\infty}
    \PP\ens{V_1>cx^2u^2 n^{\alpha-1/2}   }\min\ens{u^{2p-1},u}\mathrm du.
\end{multline}
Since for any non-negative random variable $Y$ and any $q>2$, $\sum_{n=1}^{+\infty}
n^{q-1}\PP\ens{Y>n}\leqslant \E{Y^q}$, we get the conclusion of item~\ref{itm:strong_moment}
of Theorem~\ref{thm:large_deviation_martingale_stationary_p_geq_2}.
\end{enumerate}  
 \end{proof}

The proof of Theorem~\ref{thm:large_deviation_id_distrib_p_geq_2} is completely analogous 
hence omitted.

\begin{proof}[Proof of Lemma~\ref{lem:lemme-tech}] We use the skyscrapers construction of Kakutani as 
in \cite{MR0198524}. 

Let $\pr{\ell_n}_{n\geq 1}$ be a non-increasing sequence of non-negative real 
numbers such that $\sum_{n\geq 1}\ell_n=1$. For every integer $n\geq 1$, set 
$X_n:=[0,\ell_n]\times\{n\}$ ($[0,\ell_n]$ equipped with the Lebesgue 
measure). Define then $X:= \cup_{n\geq 1}X_n$. Let $\tau$ be an ergodic 
transformation of $[0,\ell_0]$. Define an ergodic transformation 
$\theta$ on $X$ by $\theta(x,n)=(x, n+1)$ if $(x,n+1)\in X$ and by 
$\theta(x,n)=(\tau(x),0)$ otherwise.

Let $n\geq 0$. For every  $2^n\leq k\leq 2^{n+1}-1$, let $\ell_k=
\frac{\kappa}{2^{n(\gamma-1)}(n+1)^2(k+1-2^n)}$, where $\kappa$ is such that 
$\sum_{n\geq 1}\ell_n=1$.

For every $n\geq 0$ and every $(x,k)\in X$, with $2^n \leq k\leq 2^{n+1}-1$,
 set $f(x)= D(k+1-2^n)^{\gamma-1}$.  
 
 Let $0<\varepsilon\leq 1$. We have 
 \begin{equation*}
 \int_Xf (\log^+(f))^{1-\varepsilon}d\PP
 \leq CD \sum_{n\geq 0}2^{n(1-\gamma)}(n+1)^{-1-\varepsilon}\sum_{k=2^n}^{2^{n+1}-1}
 (k+1-2^n)^{\gamma-2}\leq \tilde CD \sum_{n\geq 0}(n+1)^{-1-\varepsilon} <\infty\, .
 \end{equation*}
 
 Taking $D$, large enough, we see that for every $n\geq 2$, 
 $\sum_{k=2^{n-2}}^{2^{n-1}-1}D\pr{k+1-2^{n-2}}^{\gamma-1} > 2^{n\gamma}$. 
 Hence, for that choice of $D$, we infer that 
 $f+\ldots + f\circ\theta^{2^n-1}> 
2^{n\gamma }$ on the set $\cup_{k=1}^{2^{n-1}}[0,\ell_{k+2^{n-1}-1}]
\times\{k\}$. Hence, 

\begin{equation*}
\sum_{n\geq 0} 2^{n(\gamma -1)} \PP\ens{f+\ldots + f\circ\theta^{2^n-1}> 
2^{n\gamma }}\geq \sum_{n\geq 0} 2^{n(\gamma -1)}
\sum_{k=1}^{2^{n-1}}\ell_{k+2^{n-1}-1}\geq\ c \sum_{ n\geq0}\frac1{n+1}
=+\infty\, ,
\end{equation*}
which finishes the proof.
\end{proof}

\begin{proof}[Proof of Proposition~\ref{prop:optimalite_cas_reel}] Let $\gamma=2\alpha$. Let $X$ be the 
probability space constructed in the proof of Lemma \ref{lem:lemme-tech}.  Let $\Omega_1$ be probability  space rich 
enough to support a sequence $\pr{\varepsilon_n}_{n\geq 1}$ of i.i.d. $\mathcal N(0,1)$ random 
 variables. Let $\Omega:= X\times \Omega_1$ with the product measure. 
 Let $f$ be the function satisfying the conclusion of Lemma 
 \ref{lem:lemme-tech}. For every $n\geq 1$, set 
 $X_n:= \varepsilon_n f^{1/2}\circ \theta^n$. Notice that $\pr{\varepsilon_n}
 _{n\geq 1}$ is independent from $\pr{f\circ \theta^n}_{n\geq 1}$ so that 
 $\pr{X_n}_{n\geq 1}$ is a stationary sequence of martingale differences (and ergodic). Set for every $n\geq 1$, $s_n:= \pr{ \sum_{i=1}^nf\circ \theta^i}^{1/2}$. We have, using independence, 
 \begin{align*}
 \sum_{n\geq 0} 2^{n(\alpha-1)}\PP\ens{\abs{\sum_{i=1}^{2^n}X_i}> 2^{n\alpha}}
 &= \frac2{\sqrt{2\pi}} \sum_{n\geq 0} 2^{n(2\alpha-1)} 
 \E{\int_{2^{n\alpha}/s_{2^n}}^{+\infty} {\rm e}^{-x^2/2}\, dx}\\
 &= \frac2{\sqrt{2\pi}} \int_0^{+\infty} \pr{\sum_{n\geq 0}
 2^{n(2\alpha-1)}\PP\ens{s_{2^n}^2> 2^{2n\alpha}/x^2}}\, {\rm e}^{-x^2/2}\, dx
 \\ &\geq \frac2{\sqrt{2\pi}} \pr{\int_0^1{\rm e}^{-x^2/2}\, dx}
 \sum_{n\geq 0}
 2^{n(2\alpha-1)}\PP\ens{s_{2^n}^2> 2^{2n\alpha}}=+\infty\, .
 \end{align*}
 \end{proof}

 \begin{proof}[Proof of Theorem~\ref{thm:large_deviation_orthomartingale_stationary_p_leq_2}]
We apply Theorem\ref{thm:deviation_inequality_stationary_orthomartingale_diff} with 
$r'=r$, $q=2r$ and $x:=\varepsilon \abs{\gr{n}}^{\alpha-1/r}$ in order to get 
\begin{multline}
\PP\ens{\max_{\gr{1}\imd\gr{i}\imd\gr{n}}\norm{S_{\gr{i}}}>
\varepsilon \abs{\gr{n}}^{\alpha}   }\\
\leq C\int_0^{+\infty}
    \PP\ens{\norm{X_{\gr{1}}}>\varepsilon \abs{\gr{n}}^{\alpha-1/r}u}\min\ens{u^{2r-1},u^{r-1}}\prt{1+\abs{\log u}}^{d-1}
    \mathrm du.
\end{multline}
Multiplying by $\abs{\gr{n}}^{r\alpha-2}$, summing over $\gr{n}\in \N^d$ and noticing that 
for any fixed $N$, the number of elements $\gr{k}\in \N^d$ such that $\sum_{i=1}^dk_i=N$ is 
$c_d \pr{N^{d-1}+1}$ for some constant $c_d$ depending only on $d$, we get 
\begin{multline*}
\sum_{\gr{n}\in \N^d}\abs{\gr{n}}^{r\alpha-2}\PP\ens{\max_{\gr{1}\imd\gr{i}\imd\gr{n}}\norm{S_{\gr{i}}}>
\varepsilon \abs{\gr{n}}^{\alpha}   }\\
\leq C'\int_0^{+\infty}\sum_{\gr{k}\in \N^d}2^{\pr{r\alpha-1}\sum_{i=1}^dk_i  }
    \PP\ens{\norm{X_{\gr{1}}}>\varepsilon 2^{\pr{\alpha-1/r}\sum_{i=1}^dk_i  }u}\min\ens{u^{2r-1},u^{r-1}}\prt{1+\abs{\log u}}^{d-1}
    \mathrm du\\
  \leq C''\int_0^{+\infty}\sum_{N=1}^{+\infty}2^{N\pr{r\alpha-1} }N^{d-1}
    \PP\ens{\norm{X_{\gr{1}}}>\varepsilon 2^{\pr{\alpha-1/r}N  }u}\min\ens{u^{2r-1},u^{r-1}}\prt{1+\abs{\log u}}^{d-1}
    \mathrm du
\end{multline*}
and using the fact that for any real valued random variable $Y$, 
\begin{equation}
\sum_{N=1}^{+\infty}2^{N\pr{r\alpha-1} }N^{d-1}
    \PP\ens{Y> 2^{\pr{\alpha-1/r}N  }}
 \leq K_{p,\alpha,r,d}\E{Y^{r}
\pr{\log\pr{Y}  }^{d-1}\mathbf 1\ens{Y\geq 1} 
 } ,
\end{equation}
we are reduced to prove finiteness of 
\begin{equation}
\E{\int_0^{\frac{\norm{X_{\gr{1}}}}{\varepsilon}} \pr{\frac{\norm{X_{\gr{1}}}}{u\varepsilon}}^{r}
\pr{\log\pr{\frac{\norm{X_{\gr{1}}}}{u\varepsilon}}  }^{d-1}\min\ens{u^{2r-1},u^{r-1}}\prt{1+
\abs{\log u}}^{d-1}    \mathrm du}.
\end{equation} 
 Let 
 \begin{equation}
 Y:=\int_0^{\frac{\norm{X_{\gr{1}}}}{\varepsilon}} \pr{\frac{\norm{X_{\gr{1}}}}{u\varepsilon}}^{r}
 \pr{\log \pr{\frac{\norm{X_{\gr{1}}}}{u\varepsilon}}  }^{d-1}\min\ens{u^{2r-1},u^{r-1}}
 \prt{1+\abs{\log u}}^{d-1}    \mathrm du.
 \end{equation}
Assume that $\norm{X_1}\leq \varepsilon$. Then 
\begin{equation}
Y=\int_0^{\frac{\norm{X_{\gr{1}}}}{\varepsilon}} \pr{\frac{\norm{X_{\gr{1}}}}{u\varepsilon}}^{r}
\pr{\log\pr{\frac{\norm{X_{\gr{1}}}}{u\varepsilon}}  }^{d-1}u^{2r-1}
 \prt{1+\abs{\log u}}^{d-1}    \mathrm du
\end{equation}
and the substitution $v=u/\norm{X_1}$ shows that 
\begin{equation}
Y\leq C  \norm{X_{\gr{1}}}^{2r}\pr{1-\log\pr{ \norm{X_{\gr{1}}  }}}^{d-1}.
\end{equation}
Now if we assume that $\norm{X_{\gr{1}}}>\varepsilon$, then 
\begin{multline}
Y=\int_0^{1} \pr{\frac{\norm{X_{\gr{1}}}}{u\varepsilon}}^{r}
\pr{\log\pr{\frac{\norm{X_{\gr{1}}}}{u\varepsilon}}  }^{d-1}u^{2r-1}
 \prt{1+\abs{\log u}}^{d-1}    \mathrm du\\
 +\int_1^{\frac{\norm{X_{\gr{1}}}}{\varepsilon}} \pr{\frac{\norm{X_{\gr{1}}}}{u\varepsilon}}^{r}
 \pr{\log\pr{\frac{\norm{X_{\gr{1}}}}{u\varepsilon}}  }^{d-1}u^{r-1}
 \prt{1+\abs{\log u}}^{d-1}    \mathrm du
\end{multline}
and the first term of the right-hand-side can be controlled by $C\norm{X_{\gr{1}}}^r
\pr{1+\abs{\log \norm{X_{\gr{1}}}}}^{d-1}$, while for the second, the substitution 
$t:=\log u$ and an integration by parts yield $Y\leq C\norm{X_{\gr{1}}}^r
\pr{1+\abs{\log \norm{X_{\gr{1}}}}}^{d}$. We thus got the estimate 
\begin{equation}
Y\leq C\norm{X_{\gr{1}}}^r\pr{1+\abs{\log \norm{X_{\gr{1}}}}}^{d}
\end{equation}
where $C$ depends only on $\varepsilon$, $d$ and $r$. Since $X_{\gr{1}}$ belongs to 
$\mathbb L^r\log^d\mathbb L$, we proved \eqref{eq:Baum_Katz_champs} 
and the proof of Theorem~ \ref{thm:large_deviation_orthomartingale_stationary_p_geq_2}
is finished.
 \end{proof}
 
 \begin{proof}[Proof of Theorem~\ref{thm:large_deviation_orthomartingale_stationary_p_geq_2}]
We apply Theorem\ref{thm:deviation_inequality_stationary_orthomartingale_diff} with 
$r'=2$, $q=2p$ and $x:=\varepsilon \abs{\gr{n}}^{\alpha-1/2}$ in order to get 
\begin{multline}
\PP\ens{\max_{\gr{1}\imd\gr{i}\imd\gr{n}}\norm{S_{\gr{i}}}>
\varepsilon \abs{\gr{n}}^{\alpha}   }\\
\leq C\int_0^{+\infty}
    \PP\ens{\norm{X_{\gr{1}}}>\varepsilon \abs{\gr{n}}^{\alpha-1/2}u}
    \min\ens{u^{2p-1},u}\prt{1+\abs{\log u}}^{d-1}  \mathrm du.
\end{multline}
Multiplying by $\abs{\gr{n}}^{p\pr{\alpha-1/2}-1}$, summing over $\gr{n}\in \N^d$ and noticing that 
for any fixed $N$, the number of elements $\gr{k}\in \N^d$ such that $\sum_{i=1}^dk_i=N$ is 
$c_d \pr{N^{d-1}+1}$ for some constant $c_d$ depending only on $d$, we get 
\begin{multline*}
\sum_{\gr{n}\in \N^d}\abs{\gr{n}}^{p\pr{\alpha-1/2}-1}\PP\ens{\max_{\gr{1}\imd\gr{i}\imd\gr{n}}
\norm{S_{\gr{i}}}>
\varepsilon \abs{\gr{n}}^{\alpha}   }\\
\leq C'\int_0^{+\infty}\sum_{\gr{k}\in \N^d}2^{\pr{p\pr{\alpha-1/2}-1}\sum_{i=1}^dk_i  }
    \PP\ens{\norm{X_{\gr{1}}}>\varepsilon 2^{\pr{\alpha-1/2}\sum_{i=1}^dk_i  }u}\min
    \ens{u^{2p-1},u}  \prt{1+\abs{\log u}}^{d-1}
    \mathrm du\\
  \leq C''\int_0^{+\infty}\sum_{N=1}^{+\infty}2^{Np\pr{\alpha-1/2} }N^{d-1}
    \PP\ens{\norm{X_{\gr{1}}}>\varepsilon 2^{\pr{\alpha-1/p}N  }u}\min\ens{u^{2p-1},u}\prt{1+
    \abs{\log u}}^{d-1}\mathrm du
\end{multline*}
and using the fact that for any real valued random variable $Y$, 
\begin{equation}
\sum_{N=1}^{+\infty}2^{Np\pr{\alpha-1/2} }N^{d-1}
    \PP\ens{Y> 2^{\pr{\alpha-1/2}N  }}
 \leq K_{p,\alpha,r,d}\E{Y^{p}
\pr{\log\pr{Y}  }^{d-1}\mathbf 1\ens{Y\geq 1} 
 } ,
\end{equation}
we are reduced to prove finiteness of 
\begin{equation}
\E{\int_0^{\frac{\norm{X_{\gr{1}}}}{\varepsilon}} \pr{\frac{\norm{X_{\gr{1}}}}{u\varepsilon}}^{p}
\pr{\log\pr{\frac{\norm{X_{\gr{1}}}}{u\varepsilon}}  }^{d-1}\min\ens{u^{2p-1},u}\prt{1+\abs{\log 
u}}^{d-1}    \mathrm du}.
\end{equation}   
Let 
\begin{equation}
Y:=\int_0^{\frac{\norm{X_{\gr{1}}}}{\varepsilon}} \pr{\frac{\norm{X_{\gr{1}}}}{u\varepsilon}}^{p}
\pr{\log\pr{\frac{\norm{X_{\gr{1}}}}{u\varepsilon}}  }^{d-1}\min\ens{u^{2p-1},u}\prt{1+\abs{\log 
u}}^{d-1}    \mathrm du.
\end{equation}
If $\norm{X_1}/\varepsilon\leq 1$, then 
\begin{multline}
Y=\int_0^{\frac{\norm{X_{\gr{1}}}}{\varepsilon}} \pr{\frac{\norm{X_{\gr{1}}}}{u\varepsilon}}^{p}
\pr{\log\pr{\frac{\norm{X_{\gr{1}}}}{u\varepsilon}}  }^{d-1} u^{2p-1} \prt{1+\abs{\log 
u}}^{d-1}    \mathrm du\\
\leq C\norm{X_{\gr{1}}}^{2p}\prt{1+\abs{\log 
\norm{X_{\gr{1}}}}}^{d-1}.
\end{multline}
If $\norm{X_{\gr{1}}}/\varepsilon> 1$, then 
\begin{multline}
Y=\int_0^1\pr{\frac{\norm{X_{\gr{1}}}}{u\varepsilon}}^{p}
\pr{\log\pr{\frac{\norm{X_{\gr{1}}}}{u\varepsilon}}  }^{d-1} u^{2p-1} \prt{1+\abs{\log 
u}}^{d-1}    \mathrm du\\
+\int_1^{\norm{X_{\gr{1}}}}\pr{\frac{\norm{X_{\gr{1}}}}{u\varepsilon}}^{p}
\pr{\log\pr{\frac{\norm{X_{\gr{1}}}}{u\varepsilon}}  }^{d-1} u\prt{1+\abs{\log 
u}}^{d-1}    \mathrm du.
\end{multline}
The first term can be bounded by $C\norm{X_{\gr{1}}}^p\pr{1+
\abs{\log\norm{X_{\gr{1}}}}}^{d-1}$
and for the second one, the substitution $t:=\log u$ shows that a similar upper bound can 
be given. This ends the proof of Theorem~\ref{thm:large_deviation_orthomartingale_stationary_p_geq_2}.
 \end{proof}

\textbf{Acknowledgment} The author would like to thank Christophe Cuny for giving the 
statement and proof of Lemma~\ref{lem:lemme-tech} and Proposition~\ref{prop:optimalite_cas_reel}.
 

\def\polhk\#1{\setbox0=\hbox{\#1}{{\o}oalign{\hidewidth
  \lower1.5ex\hbox{`}\hidewidth\crcr\unhbox0}}}\def\cprime{$'$}
  \def\polhk#1{\setbox0=\hbox{#1}{\ooalign{\hidewidth
  \lower1.5ex\hbox{`}\hidewidth\crcr\unhbox0}}} \def\cprime{$'$}
\providecommand{\bysame}{\leavevmode\hbox to3em{\hrulefill}\thinspace}
\providecommand{\MR}{\relax\ifhmode\unskip\space\fi MR }
\providecommand{\MRhref}[2]{%
  \href{http://www.ams.org/mathscinet-getitem?mr=#1}{#2}
}

\end{document}